\documentclass[reqno,tbtags]{amsproc}
\usepackage{amssymb}
\usepackage{epsfig}
\usepackage{chicago}
\usepackage{euscript}
\usepackage{eufrak}
\newtheorem{theorem}{Theorem}[section]
\newtheorem{corollary}{Corollary}[section]
\newtheorem{lemma}{Lemma}[section]

\theoremstyle{definition}

\newtheorem{remark}{Remark}[section]

\def\ANP{\emph{Annals of Probability\/}}

\def\IME{\emph{In\-su\-ran\-ce: Ma\-the\-ma\-tics and Eco\-no\-mics\/}}

\def\LMJ{\emph{Lith. Math. J.\/}}

\def\TPA{\emph{Theory Probab. Appl.}}
\def\JRSS{\emph{Jour\-nal of the Royal Sta\-tist. Soc., Ser. B\/}}

\numberwithin{equation}{section}
\newcommand{\NomiN}[1]{\varrho_{#1}}
\newcommand{\DeNomiN}[1]{\vartheta_{#1}}
\def\eqOK{=}
\def\myEq{=}
\def\eqSic{=}
\newcommand{\Integral}[1]{\mathcal{I}_{\,#1}}
\newcommand{\Integrall}[1]{\mathcal{T}_{\,#1}}
\newcommand{\Jntegral}[1]{\mathcal{J}_{\,#1}}
\newcommand{\IntegraL}[2]{\mathcal{I}_{\,#1,#2}}
\newcommand{\IntegralL}[2]{\mathcal{T}_{\,#1,#2}}
\def\mcA{\mathcal{A}}
\def\mcB{\mathcal{B}}
\def\mcC{\mathcal{C}}
\newcommand{\OneVar}{\mathcal{X}}
\newcommand{\TwoVar}{\mathcal{V}}
\newcommand{\MainApprox}[1]{\mathcal{A}_{#1}}
\newcommand{\SeqApprox}[2]{\mathcal{A}^{\langle #2\rangle}_{#1}}

\newcommand{\MainRemTerm}[1]{\mathcal{R}_{#1}}
\newcommand{\SeqRemTerm}[2]{\mathcal{R}^{\langle #2\rangle}_{#1}}
\newcommand{\IntOnE}{\EuScript{S}}
\newcommand{\IntOne}[1]{\EuScript{S}^{[#1]}}
\newcommand{\IntOneHaT}{\hat{\EuScript{S}}}
\newcommand{\IntOneHat}[1]{\hat{\EuScript{S}}^{[#1]}}
\newcommand{\IntTwO}{\EuScript{G}}
\newcommand{\IntTwo}[1]{\EuScript{G}^{[#1]}}
\newcommand{\IntTwoHat}{\hat{\EuScript{G}}}
\newcommand{\cK}[1]{K^{#1}_{c}\hskip 0.5pt}
\def\hatM{M_{{u+\cS}v,n}}
\def\hatL{L_{{u+\cS}v,n}}
\def\hatR{R_{{u+\cS}v,n}}
\def\XatM{M_{{u+c}v,n}}
\def\XatL{L_{{u+c}v,n}}
\def\XatR{R_{{u+c}v,n}}
\def\cS{c^{*}}
\newcommand{\timeR}{\Upsilon}
\renewcommand{\P}{\mathsf{P}}
\newcommand{\p}{\mathsf{p}}
\newcommand{\R}{\mathsf{R}}
\newcommand{\D}{\mathsf{D}}
\newcommand{\E}{\mathsf{E}}
\newcommand{\Uex}{U_{\epsilon,x}}
\newcommand{\EnOne}{N_{\epsilon}}
\newcommand{\Y}[1]{Y_{#1}}
\newcommand{\T}[1]{T_{#1}}
\newcommand{\X}[1]{X_{#1}}
\newcommand{\UGauss}[2]{\varPhi_{\left({#1},{#2}\right)}}
\newcommand{\Ugauss}[2]{\varphi_{\left({#1},{#2}\right)}}
\newcommand{\homN}[1]{N_{#1}}
\newcommand{\homV}[1]{V_{#1}}

\newcommand{\muIG}{\mu}
\newcommand{\lambdaIG}{\lambda}
\def\mmu{\hat{\mu}}
\def\paramY{\alpha}
\def\paramT{\beta}
\newcommand{\Int}[2]{{\mathcal{#2}}_{#1}}
\begin{document}
%%%%%%%%%%%%%%%%%%%%%%%%%%%%%%%%%%%%%%%%%%%%%%%%%%%%%%%%%%%%%%%%%%%%%%%%%%%%%%%%%%%%%%%%
\author[Vsevolod K. Malinovskii]{\Large Vsevolod K. Malinovskii}


\address{Central Economics and Mathematics Institute (CEMI) of Russian Academy of Science,
117418, Nakhimovskiy prosp., 47, Moscow, Russia}

\email{malinov@orc.ru, malinov@mi.ras.ru}

\urladdr{http://www.actlab.ru}

\title[ON THE TIME OF FIRST LEVEL CROSSING]{ON THE TIME OF FIRST LEVEL CROSSING AND INVERSE GAUSSIAN DISTRIBUTION}

\maketitle

\begin{abstract}
We propose a new approximation for the distribution of the time of the first
level $u$ crossing by the random process $\homV{s}-cs$, where $\homV{s}$,
$s>0$, is compound renewal process and $c>0$. It is competitive with respect to
existing approximations, particularly in the region around the critical point
$c=\cS$ which separates processes with positive and negative drifts. This
approximation is tightly related to inverse Gaussian distributions.
\end{abstract}

%%%%%%%%%%%%%%%%%%%%%%%%%%%%%%%%%%%%%%%%%%%%%%%%%%%%%%%%%%%%%%%%%%%%%%%%%%%%%%%%%%%%%%%%
\section{Introduction and main result}\label{dsrgherhr}
%%%%%%%%%%%%%%%%%%%%%%%%%%%%%%%%%%%%%%%%%%%%%%%%%%%%%%%%%%%%%%%%%%%%%%%%%%%%%%%%%%%%%%%%

The inverse Gaussian distribution has probability density function (p.d.f.)
\begin{equation}\label{wqrdtgrehr}
f\left(x;\muIG,\lambdaIG,-\tfrac{1}{2}\right)=
\frac{\lambdaIG^{1/2}}{\sqrt{2\pi}}\,x^{-3/2}\exp\left\{-\frac{\lambdaIG(x-\muIG)^2}{2\muIG^2
x}\right\},
\end{equation}
where $x$, $\muIG$, and $\lambdaIG$ are positive. Parameter $\lambdaIG$ is
called shape parameter, and $\muIG$ is called mean parameter.

In the study of this distribution, paramount is finding explicit expression
\begin{multline}\label{werthrmrt}
F\left(x;\muIG,\lambdaIG,-\tfrac{1}{2}\right)=\int_{0}^{x}f\left(z;\muIG,\lambdaIG,-\tfrac{1}{2}\right)dz
\\
\eqOK\UGauss{0}{1}\left(\sqrt{\frac{\lambdaIG}{x}}\left(\frac{x}{\muIG}-1\right)\right)
+\exp\bigg\{\frac{2\lambdaIG}{\muIG}\bigg\}\,\UGauss{0}{1}
\left(-\sqrt{\frac{\lambdaIG}{x}}\left(\frac{x}{\muIG}+1\right)\right)
\end{multline}
for cumulative distribution function (c.d.f.) corresponding to p.d.f.
\eqref{wqrdtgrehr}; by $\UGauss{0}{1}(x)$ and $\Ugauss{0}{1}(x)$ we denote
c.d.f. and p.d.f. of a standard normal distribution.

By $f_{\T{1}}(t)$ and $f_{T}(t)$ we denote p.d.f. of positive random variable
$\T{1}$, and of positive random variables $\T{i}\overset{d}{=}T$,
$i=2,3,\dots$, all distributed identically. The random variable $\T{1}$ is the
time between starting time zero and time of the first renewal, and the random
variables $\T{i}$ are inter-renewal times. By $f_{Y}(t)$ we denote p.d.f. of
positive random variables $\Y{i}\overset{d}{=}Y$, $i=1,2,\dots$, all
distributed identically. The random variables $\Y{i}$ are called jump sizes,
and the jumps occur only in the moments of renewals. Throughout the entire
presentation, p.d.f. $f_{T}(y)$ and $f_{Y}(y)$ are assumed bounded from above
by a finite constant. Having assumed that $\T{1}$, i.i.d.
$\T{i}\overset{d}{=}T$, $i=2,3,\dots$, i.i.d. $\Y{i}\overset{d}{=}Y$,
$i=1,2,\dots$, are all mutually independent, we are within renewal model, where
the distribution of the first interval $\T{1}$ may be different from the
distribution of the other interclaim intervals, i.e., from the distribution of
$T$.

Compound renewal process with time $s\geqslant 0$ is
\begin{equation*}
\homV{s}=\sum_{i=1}^{\homN{s}}\Y{i},
\end{equation*}
or $0$, if $\homN{s}=0$ (or $\T{1}>s$), where
$\homN{s}=\max\left\{n>0:\sum_{i=1}^{n}\T{i}\leqslant s\right\}$, or $0$, if
$\T{1}>s$. The random variable
\begin{equation}\label{wertrjnmt}
\timeR=\inf\left\{s>0:\homV{s}-cs>u\right\},
\end{equation}
or $+\infty$, as $\homV{s}-cs\leqslant u$ for all $s>0$, is the time of the
first level $u$ crossing by the process $\homV{s}-cs$.

It is easily seen that for $t>0$
\begin{equation}\label{ewrktulertye}
\P\{\timeR\leqslant
t\}=\int_{0}^{t}\P\{u+cv-\Y{1}<0\}f_{\T{1}}(v)dv+\int_{0}^{t}\P\{v<\timeR\leqslant
t\mid\T{1}=v\}f_{\T{1}}(v)dv.
\end{equation}

This distribution of $\timeR$ appears in many branches of applied probability,
including risk and queueing theories, and was considered by many authors. For
it, there are many closed-form formulas and approximations, derived by
different techniques. The goal of this paper is to get the approximation that
involves inverse Gaussian distribution, and seems new. Remarkable is that it is
derived under a set of conditions similar to those usually imposed in the
common local central limit theorem.

\begin{figure}[t]
\includegraphics[scale=0.8]{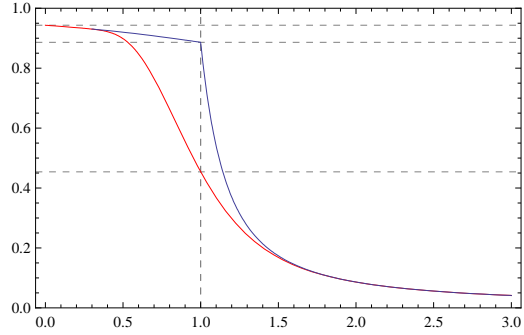}
\caption{\small Graphs ($X$-axis is $c$) of the functions
$\Int{\infty}{M}(u,c,0)$ and $\Int{t}{M}(u,c,0)$, as $M=1$, $D^2=6$, $t=100$,
$u=15$. Horizontal lines are $\Int{\infty}{M}(u,0,0)=0.943$,
$\Int{\infty}{M}(u,\cS,0)=0.886$, and $\Int{t}{M}(u,\cS,0)=0.454$, where
$\cS=\E{Y}/\E{T}$.}\label{sdrtgerher}
\end{figure}

Set $M={\E{T}}/{\E{Y}}$, $D^2=((\E{T})^2\D{Y}+(\E{Y})^2\D{T})/(\E{Y})^3$, and
introduce
\begin{equation*}
\Int{t}{M}(u,c,v)=\int_{0}^{\frac{c(t-v)}{u+cv}}\frac{1}{1+x}
\,\Ugauss{cM(1+x)}{\frac{c^2D^2(1+x)}{u+cv}}(x)dx.
\end{equation*}
For $\cS=\frac{1}{M}$, $\lambdaIG=\frac{u+cv}{c^2D^2}>0$,
$\muIG=\frac{1}{1-cM}$, and $\mmu=-\muIG=\frac{1}{cM-1}$, using
\eqref{wqrdtgrehr} and \eqref{werthrmrt}, we have
\begin{equation*}
\begin{aligned}
\Int{t}{M}(u,c,v)&=\begin{cases}
F\big(x;\muIG,\lambdaIG,-\tfrac{1}{2}\big)\Big|_{x=1}^{\frac{c(t-v)}{u+cv}+1},&0<c\leqslant\cS,
\\[6pt]
\exp\big\{-2\,\tfrac{\lambdaIG}{\mmu}\big\}
F\big(x;\mmu,\lambdaIG,-\tfrac{1}{2}\big)\Big|_{x=1}^{\frac{c(t-v)}{u+cv}+1},&
c\geqslant\cS
\end{cases}
\\
&\eqOK
\bigg[\UGauss{0}{1}\Big(\tfrac{\sqrt{u+cv}}{cD\sqrt{x}}\big(x(1-cM)-1\big)\Big)
\\[-4pt]
&\hskip
28pt+\exp\Big\{2\frac{u+cv}{c^2D^2}(1-cM)\Big\}\UGauss{0}{1}\Big(-\tfrac{\sqrt{u+cv}}{cD
\sqrt{x}}\big(x(1-cM)+1\big)\Big)\bigg]\bigg|_{x=1}^{\frac{c(t-v)}{u+cv}+1}.
\end{aligned}
\end{equation*}

\begin{theorem}\label{srdthjrf}
In the above model, let p.d.f. $f_{T}(y)$ and $f_{Y}(y)$ be bounded from above
by a finite constant, $D^2>0$, $\E({T}^3)<\infty$, $\E({Y}^3)<\infty$. Then for
$c>0$, for fixed $0<v<t$ we have
\begin{equation}\label{werhrjet}
\sup_{t>v}\Big|\,\P\{v<\timeR\leqslant t\mid\T{1}=v\}-\Int{t}{M}(u,c,v)\Big|
=\underline{O}\bigg(\frac{\ln(u+cv)}{u+cv}\bigg),
\end{equation}
as $u+cv\to\infty$.
\end{theorem}

As soon as the distribution of $T_1$ is specified, the similar results for
$\P\{\timeR\leqslant t\}$ are straightforward from Theorem \ref{srdthjrf}. In
particular, for $\T{1}$ exponential with parameter $\paramT$
\begin{equation*}
\P\{\timeR\leqslant t\}=\paramT\int_{0}^{t}\P\{u+cv-\Y{1}<0\}e^{-\paramT
v}dv+\paramT\int_{0}^{t}\P\{v<\timeR\leqslant t\mid\T{1}=v\}e^{-\paramT v}dv.
\end{equation*}
If $Y$ is exponential with parameter $\paramY$, we have
\begin{multline*}
\P\{\timeR\leqslant t\}=\paramT\int_{0}^{t}e^{-\paramY(u+cv)}e^{-\paramT
v}dv+\paramT\int_{0}^{t}\P\{v<\timeR\leqslant t\mid\T{1}=v\}e^{-\paramT v}dv
\\
=\frac{\paramT e^{-\paramY
u}}{\paramT+c\paramY}\big(1-e^{-(\paramT+c\paramY)t}\big)+\paramT\int_{0}^{t}
\P\{v<\timeR\leqslant t\mid\T{1}=v\}e^{-\paramT v}dv.
\end{multline*}

%%%%%%%%%%%%%%%%%%%%%%%%%%%%%%%%%%%%%%%%%%%%%%%%%%%%%%%%%%%%%%%%%%%%%%%%%%%%%%%%%%%%%%%%
\section{Closed-form expression using convolutions}\label{wertyjtkt}
%%%%%%%%%%%%%%%%%%%%%%%%%%%%%%%%%%%%%%%%%%%%%%%%%%%%%%%%%%%%%%%%%%%%%%%%%%%%%%%%%%%%%%%%

The following result is a modification of a result in \citeNP{[Borovkov Dickson
2008]}.

\begin{theorem}\label{23456uy7}
For $M(s)=\inf\{k\geqslant 1:\sum_{i=1}^{k}Y_{i}>s\}-1$, we have
\begin{equation}\label{qwretgjhmg}
\P\{v<\timeR\leqslant
t\mid\T{1}=v\}=\int_{v}^{t}\dfrac{u+cv}{u+cz}\sum_{n=1}^{\infty}
\P\big\{M(u+cz)=n\big\}f_{T}^{*n}(z-v)dz,
\end{equation}
and
\begin{multline}\label{3eryjkuy}
\P\{\timeR\leqslant t\}=\int_{0}^{t}\P\{u+cv-Y_{1}<0\}f_{\T{1}}(v)dv
\\[0pt]
+\int_{0}^{t}\left[\int_{v}^{t}\frac{cv+u}{u+cs}
\sum_{n=1}^{\infty}\P\big\{M(u+cs)=n\big\}f_{T}^{*n}(s-v)ds\right]f_{\T{1}}(v)dv.
\end{multline}
\end{theorem}

\begin{corollary}[Theorem 1 in \citeNP{[Borovkov Dickson 2008]}]\label{etyjhtrjrt}
For $Y$ exponential with parameter $\paramY$, we have
\begin{multline}\label{wedteyjtkj}
\P\{\timeR\leqslant t\}=\int_{0}^{t}e^{-\paramY(u+cs)}\left[f_{\T{1}}(s)
+\frac{1}{u+cs}\sum_{n=1}^{\infty}\frac{(\paramY(u+cs))^n}{n!}\right.
\\[0pt]
\left.\times\int_{0}^{s}(u+cv)f_{T}^{*n}(s-v)f_{\T{1}}(v)dv\right]ds.
\end{multline}
\end{corollary}

\begin{proof}[Proof of Theorem~\ref{23456uy7}]\label{sryehr}
The main idea of this proof is to change jumps direction from ``toward the
boundary'' to ``away from the boundary'' and then use Kendall's identity. We
set $\T{1}=v$ and write
\begin{equation}\label{edrtfyreh}
\timeR=(\sigma-u)/c,
\end{equation}
where $\sigma$ is the crossing time of the lower level $-(v+u/c)$ by the
process $Z(s)=\sum_{k\leqslant M(s)}T_k-s/c$ (here $M(s)=\inf\{k\geqslant
1:\sum_{i=1}^{k}Y_{i}>s\}-1$), which is a skip-free in the negative
direction\footnote{Recall that this means that this process has no negative
jumps and its increments are stationary independent.} L{\'e}vy process.

The Kendall's identity writes as (see \citeNP{[Borovkov Dickson 2008]}):
\begin{equation*}
p_{\sigma}(s)=\frac{v+u/c}{s}p_{Z(s)}(-(v+u/c)),
\end{equation*}
where
\begin{equation*}
p_{Z(s)}(-(v+u/c))=\sum_{n=1}^{\infty}\P\big\{M(s)=n\big\}f_{T}^{*n}(-(v+u/c)+s/c).
\end{equation*}

According to \eqref{edrtfyreh}, we have $p_{\timeR}(t\mid
v)=cp_{\sigma}(u+ct)$. We observe that $\timeR\geqslant \T{1}$ holds always and
write
\begin{multline*}
\P\{\timeR\leqslant t\}=\P\{\T{1}=\timeR\leqslant t\}+\P\{\T{1}<\timeR\leqslant
t\}
\\
=\int_{0}^{t}\P\{u+cv-Y_1<0\}f_{\T{1}}(v)dv+\int_{0}^{t}\P\{v<\timeR\leqslant
t\mid \T{1}=v\}f_{\T{1}}(v)dv
\\
=\int_{0}^{t}\P\{u+cv-Y_1<0\}f_{\T{1}}(v)dv+\int_{0}^{t}\left[\int_{v}^{t}p_{\timeR}(z\mid
v)dz\right]f_{\T{1}}(v)dv
\\
=\int_{0}^{t}\P\{u+cv-Y_1<0\}f_{\T{1}}(v)dv
+\int_{0}^{t}\left[\int_{v}^{t}cp_{\sigma}(u+cs)ds\right]f_{\T{1}}(v)dv
\end{multline*}
with
\begin{equation*}
cp_{\sigma}(u+cs)=\frac{cv+u}{u+cs}\sum_{n=1}^{\infty}\P\big\{M(u+cs)=n\big\}f_{T}^{*n}(s-v),
\end{equation*}
which is required.
\end{proof}

\begin{proof}[Proof of Corollary~\ref{etyjhtrjrt}]
For $\Y{i}\overset{d}{=}Y$, $i=1,2,\dots$, exponential with parameter
$\paramY$, we have
\begin{equation*}
\P\big\{M(u+cs)=n\big\}=e^{-\paramY(u+cs)}\frac{(\paramY(u+cs))^n}{n!},\quad
n=1,2,\dots,
\end{equation*}
$\P\{u+cv-\Y{1}<0\}=e^{-\paramY(u+cv)}$, and equation \eqref{3eryjkuy} turns
into \eqref{wedteyjtkj}, as required.
\end{proof}

%%%%%%%%%%%%%%%%%%%%%%%%%%%%%%%%%%%%%%%%%%%%%%%%%%%%%%%%%%%%%%%%%%%%%%%%%%%%%%%%%%%%%%%%
\section{Proof of Theorem \ref{srdthjrf}}\label{sdfyguklyi}
%%%%%%%%%%%%%%%%%%%%%%%%%%%%%%%%%%%%%%%%%%%%%%%%%%%%%%%%%%%%%%%%%%%%%%%%%%%%%%%%%%%%%%%%

In the sequel, let $K$, $K_{1}$, $K_{2}$, etc., be ``sufficiently large''
positive constants, and $\epsilon$, $\epsilon_{1}$, $\epsilon_{2}$, etc., be
``sufficiently small'' positive constants. All of them do not depend on
summation and integration variables, such as $n$, $y$, $z$, etc., and possibly
are different in different equations.

We put $y=z-v$ in \eqref{qwretgjhmg} and rewrite it as
\begin{align}
\P\{v<\timeR\leqslant t\mid\T{1}=v\}&=\int_{0}^{t-v}\dfrac{u+cv}{u+cv+cy}\,\,
\p_{\big\{\sum_{i=2}^{M(u+cv+cy)+1}\T{i}\big\}}(y)dy\notag
\\
&=\sum_{n=1}^{\infty}\int_{0}^{t-v}\frac{u+cv}{u+cv+cy}
\int_{0}^{u+cv+cy}\P\left\{\Y{n+1}>z\right\}\notag
\\[2pt]
&\hskip 90pt\times f_{Y}^{*n}(u+cv+cy-z)f_{T}^{*n}(y)dydz.\label{wertyrjhn}
\end{align}
Bearing in mind that $\T{i}$, $i=1,2,\dots$ and $\Y{i}$, $i=1,2,\dots$ are
mutually independent, the second equality in \eqref{wertyrjhn} holds true since
\begin{multline*}
\P\{M(u+cv+cy)=n\}=\P\bigg\{\sum_{i=1}^{n}\Y{i}\leqslant
u+cv+cy<\sum_{i=1}^{n+1}\Y{i}\bigg\}
\\
=\int_{0}^{u+cv+cy}f^{*n}_{Y}(u+cv+cy-z)\P\{\Y{n+1}>z\}dz.
\end{multline*}

The proof of Theorem \ref{srdthjrf} consists of several steps. The first and
the third steps are elimination of the terms that have little impact in
\eqref{wertyrjhn}; it may be called preparation of \eqref{wertyrjhn} for
further analysis. The former is elimination of those terms that correspond to
small $n$, i.e., to such $n$ that the event $\{M(u+cv+cy)=n\}$ has a small
probability, provided that $u+cv+cy$ is large. The latter is elimination of the
terms containing $z$, i.e., defect of the random walk $\sum_{i=1}^{n}Y_i$,
$n=1,2,\dots$, as it nearly crosses the high level $u+cv+cy$. In the first
step, we use the bounds for large deviations of sums of i.i.d. random
variables, like in \citeNP{[Nagaev 1965]}. In the third step, we apply the
Taylor formula to the normal p.d.f.

The second step yields the main term of approximation and the corresponding
remainder term in a raw form. That is made by means of applying standard
non-uniform Berry-Esseen bounds in local CLT formulated in
Section~\ref{srdthrj} to the product of $f^{*n}_{Y}$ and $f^{*n}_{T}$ in
\eqref{wertyrjhn}. The fourth step consists in investigation of the asymptotic
behavior of core components in the remainder terms which appear all over the
proof. The fifth step is the simplification of the main term of approximation,
up to the terms of required order of magnitude. It relies on a standard
estimation technique developed on the fourth step.

%%%%%%%%%%%%%%%%%%%%%%%%%%%%%%%%%%%%%%%%%%%%%%%%%%%%%%%%%%%%%%%%%%%%%%%%%%%%%%%%%%%%%%%%
\subsection{Step 1: reducing of the area of summation}\label{sqrdstfghj}
%%%%%%%%%%%%%%%%%%%%%%%%%%%%%%%%%%%%%%%%%%%%%%%%%%%%%%%%%%%%%%%%%%%%%%%%%%%%%%%%%%%%%%%%

Let us rewrite \eqref{wertyrjhn} as
\begin{equation*}
\begin{aligned}
\P\{v<\timeR\leqslant t\mid\T{1}=v\}& =\int_{0}^{t-v}\dfrac{u+cv}{u+cv+cy}\,\,
\p_{\big\{\sum_{i=2}^{M(u+cv+cy)+1}\T{i}\big\}}(y)dy
\\
&=\sum_{n=1}^{\infty}\int_{0}^{t-v}\dfrac{u+cv}{u+cv+cy}\,\,\P\{M(u+cv+cy)=n\}\,
\p_{\big\{\sum_{i=2}^{n+1}\T{i}\big\}}(y)dy,
\end{aligned}
\end{equation*}
select $\EnOne=\epsilon\,(u+cv)$, where $0<\epsilon<1$, split the sum
$\sum_{1}^{\infty}=\sum_{1}^{\EnOne}+\sum_{\EnOne}^{\infty}$ and show that the
first summand may be omitted within the required accuracy of approximation.

Note that $(u+cv)/(u+cv+cy)<1$ for $y>0$. Since
\begin{equation*}
\P\{M(u+cv+cy)=n\}=\P\bigg\{\sum_{i=1}^{n}\Y{i}\leqslant
u+cv+cy<\sum_{i=1}^{n+1}\Y{i}\bigg\}\leqslant\P\bigg\{\sum_{i=1}^{n+1}\Y{i}>u+cv+cy\bigg\},
\end{equation*}
we have
\begin{equation*}
\begin{aligned}
&\sum_{n=1}^{\EnOne}\int_{0}^{t-v}\P\{M(u+cv+cy)=n\}
\p_{\big\{\sum_{i=2}^{n+1}\T{i}\big\}}(y)dy
\\
&\leqslant\sum_{n=1}^{\EnOne}\int_{0}^{t-v}\P\bigg\{\sum_{i=1}^{n+1}\Y{i}>u+cv+cy\bigg\}
\p_{\big\{\sum_{i=2}^{n+1}\T{i}\big\}}(y)dy
\\
&\leqslant\sum_{n=1}^{\EnOne}\P\bigg\{Y_1+\sum_{i=2}^{n+1}\X{i}>u+cv\bigg\}.
\end{aligned}
\end{equation*}

For standardized i.i.d. random variables $\tilde{X}_i=(X_i-\E{X})/\sqrt{\D{X}}$, $i=1,2,\dots$, we have
\begin{equation*}
\P\bigg\{\sum_{i=2}^{n+1}\X{i}>u+cv\bigg\}
=\P\bigg\{\sum_{i=2}^{n+1}\tilde{\X{i}}>\frac{u+cv-n\E{X}}{\sqrt{\D{X}}}\bigg\}.
\end{equation*}
It is easily seen that the inequality
\begin{equation*}
\frac{u+cv-n\E{X}}{\sqrt{\D{X}}}>K\sqrt{n\ln(n)}
\end{equation*}
holds true for all $n\leqslant\EnOne$, and by Lemma~\ref{q1werhewef} we have
\begin{equation*}
\P\bigg\{\sum_{i=2}^{n+1}\tilde{\X{i}}>\frac{u+cv-n\E{X}}{\sqrt{\D{X}}}\bigg\}
\leqslant K\frac{n}{(u+cv-n)^3}.
\end{equation*}
Using simple estimates\footnote{Use, e.g., the inequality
$\P\{\xi_1+\xi_2>x\}\leqslant\P\{\xi_1>x/2\}+\P\{\xi_2>x/2\}$.}, we have
\begin{multline*}
\sum_{n=1}^{\EnOne}\P\bigg\{Y_1+\sum_{i=2}^{n+1}\X{i}>u+cv\bigg\}\leqslant
K\,\sum_{n=1}^{\EnOne}\frac{n}{(u+cv-n)^3}
=\frac{K}{(u+cv)^3}\,\sum_{n=1}^{\EnOne}\frac{n}{(1-\frac{n}{u+cv})^3}
\\
\leqslant\frac{K_1}{(u+cv)^3}\,\sum_{n=1}^{\EnOne}n=\frac{K_1}{(u+cv)^3}\,\frac{\EnOne(\EnOne+1)}{2}
=\underline{O}((u+cv)^{-1}),
\end{multline*}
as $u+cv\to\infty$, as required.

%%%%%%%%%%%%%%%%%%%%%%%%%%%%%%%%%%%%%%%%%%%%%%%%%%%%%%%%%%%%%%%%%%%%%%%%%%%%%%%%%%%%%%%%
\subsection{Step 2: application of CLT}\label{qweryjky}
%%%%%%%%%%%%%%%%%%%%%%%%%%%%%%%%%%%%%%%%%%%%%%%%%%%%%%%%%%%%%%%%%%%%%%%%%%%%%%%%%%%%%%%%

For i.i.d. random vectors $\xi_i=(\tilde{Y}_i,\tilde{T}_i)\in\R^{2}$ with
standardized independent components $\tilde{Y}_i=(Y_i-\E{Y})/\sqrt{\D{Y}}$ and
$\tilde{T}_i=(T_i-\E{T})/\sqrt{\D{T}}$, $i=1,2,\dots$, let us apply
Theorem~\ref{w4e5y46yu43}. Bearing in mind that
\begin{equation*}
f_{T}^{*n}(x)=\frac{1}{\sqrt{n\D{T}}}\,p_{n^{-1/2}\sum_{i=1}^{n}\tilde{T}_i}
\left(\tfrac{x-n\E{T}}{\sqrt{n\D{T}}}\right),\quad
f_{Y}^{*n}(x)=\frac{1}{\sqrt{n\D{Y}}}\,p_{n^{-1/2}\sum_{i=1}^{n}\tilde{Y}_i}
\left(\tfrac{x-n\E{Y}}{\sqrt{n\D{Y}}}\right),
\end{equation*}
we have from
Theorem~\ref{w4e5y46yu43}
\begin{equation}\label{werkyhjrtjr}
\Big|\,\P\{v<\timeR\leqslant
t\mid\T{1}=v\}-\MainApprox{t}(u,c\mid\T{1}=v)\Big|\leqslant\MainRemTerm{t}(u,c\mid\T{1}=v),
\end{equation}
where $c,u>0$, $0<v<t$ and
\begin{equation*}
\begin{aligned}
\MainApprox{t}(u,c\mid\T{1}=v)&=\sum_{n=\EnOne}^{\infty}\int_{0}^{t-v}\frac{u+cv}{u+cv+cy}
\int_{0}^{u+cv+cy}\P\left\{\Y{n+1}>z\right\}
\\
&\hskip
-40pt\times\frac{1}{\sqrt{n\D{Y}}}\,\Ugauss{0}{1}\left(\tfrac{u+cv+cy-z-n\E{Y}}{\sqrt{n\D{Y}}}\right)
\frac{1}{\sqrt{n\D{T}}}\,\Ugauss{0}{1}\left(\tfrac{y-n\E{T}}{\sqrt{n\D{T}}}\right)\,dydz,
\end{aligned}
\end{equation*}
\begin{equation*}
\begin{aligned}
\MainRemTerm{t}(u,c\mid\T{1}=v)&=K\sum_{n=\EnOne}^{\infty}\int_{0}^{t-v}\frac{u+cv}{u+cv+cy}
\int_{0}^{u+cv+cy}\P\left\{\Y{n+1}>z\right\}
\\
&\hskip
-40pt\times\frac{1}{n^{3/2}}\bigg(1+\bigg[\bigg(\frac{u+cv+cy-z-n\E{Y}}{\sqrt{n\D{Y}}}\bigg)^2
+\bigg(\frac{y-n\E{T}}{\sqrt{n\D{T}}}\bigg)^2\bigg]^{1/2}\bigg)^{-3}\,dydz.
\end{aligned}
\end{equation*}
\smallskip

\begin{remark}\label{wqeyhjr}
To get the approximation \eqref{werkyhjrtjr}, to the product
$f_{Y}^{*n}(u+cv+cy-z)f_{T}^{*n}(y)$ we applied Theorem~\ref{w4e5y46yu43},
which is the Berry-Esseen bounds in two-di\-men\-si\-onal local CLT. Instead,
we could apply the Berry-Esseen bounds in one-dimensional local CLT to each of
these factors, separately. We preferred to use Theorem~\ref{w4e5y46yu43} to get
the remainder term $\MainRemTerm{t}(u,c\mid\T{1}=v)$ in a form better suited
for the further analysis.
\end{remark}

%%%%%%%%%%%%%%%%%%%%%%%%%%%%%%%%%%%%%%%%%%%%%%%%%%%%%%%%%%%%%%%%%%%%%%%%%%%%%%%%%%%%%%%%
\subsection{Step 3: bringing the approximation to a convenient form}\label{wertwserewye}
%%%%%%%%%%%%%%%%%%%%%%%%%%%%%%%%%%%%%%%%%%%%%%%%%%%%%%%%%%%%%%%%%%%%%%%%%%%%%%%%%%%%%%%%

We do this in several steps. Major objective is simplification of the main term
$\MainApprox{t}(u,c\mid\T{1}=v)$ and verification that the remainder term
$\MainRemTerm{t}(u,c\mid\T{1}=v)$ is of required order of smallness.

%%%%%%%%%%%%%%%%%%%%%%%%%%%%%%%%%%%%%%%%%%%%%%%%%%%%%%%%%%%%%%%%%%%%%%%%%%%%%%%%%%%%%%%%
\subsection*{Change of variables}\label{dgfjryuhjrt}
%%%%%%%%%%%%%%%%%%%%%%%%%%%%%%%%%%%%%%%%%%%%%%%%%%%%%%%%%%%%%%%%%%%%%%%%%%%%%%%%%%%%%%%%

Put $x=c\,y/(u+cv)$, $dx=c\,dy/(u+cv)$. For
\begin{equation*}
\mathcal{Y}_{n,z}(u+cv,x)=\dfrac{(u+cv)(1+x)-z-n\E{Y}}{\sqrt{n\D{Y}}},\hskip
6pt\mathcal{T}_n(u+cv,x)=\dfrac{(u+cv)x/c-n\E{T}}{\sqrt{n\D{T}}}
\end{equation*}
we have\footnote{We bear in mind that $Y_{n+1}\overset{d}{=}Y$ and that
$cy=(u+cv)x$, $cdy=(u+cv)dx$.}
\begin{multline}\label{wertyhjty}
\MainApprox{t}(u,c\mid\T{1}=v)=\frac{u+cv}{c\sqrt{\D{T}\D{Y}}}
\,\sum_{n=\EnOne}^{\infty}n^{-1}\,\int_{0}^{\frac{c(t-v)}{u+cv}}
\frac{1}{1+x}\int_{0}^{(u+cv)(1+x)}\P\left\{Y>z\right\}
\\
\times\Ugauss{0}{1}\big(\mathcal{Y}_{n,z}(u+cv,x)\big)
\,\Ugauss{0}{1}\big(\mathcal{T}_n(u+cv,x)\big)\,dxdz
\end{multline}
and
\begin{multline}\label{asdfghd}
\MainRemTerm{t}(u,c\mid\T{1}=v)=K(u+cv)\sum_{n=\EnOne}^{\infty}n^{-3/2}
\int_{0}^{\frac{c(t-v)}{u+cv}}\frac{1}{1+x}
\int_{0}^{(u+cv)(1+x)}\P\left\{Y>z\right\}\,
\\
\times\big(1+\big[\mathcal{Y}^2_{n,z}(u+cv,x)
+\mathcal{T}^2_n(u+cv,x)\big]^{1/2}\big)^{-3}\,dxdz.
\end{multline}

%%%%%%%%%%%%%%%%%%%%%%%%%%%%%%%%%%%%%%%%%%%%%%%%%%%%%%%%%%%%%%%%%%%%%%%%%%%%%%%%%%%%%%%%
\subsection*{Use of fundamental identities of Section~\ref{werwhteyjn}}\label{wertrjtyh}
%%%%%%%%%%%%%%%%%%%%%%%%%%%%%%%%%%%%%%%%%%%%%%%%%%%%%%%%%%%%%%%%%%%%%%%%%%%%%%%%%%%%%%%%

Denoting
\begin{equation*}
B_1=(\E{T})^2\D{Y}+(\E{Y})^2\D{T},\ B_2=\E{Y}\D{T},\ B_3=\E{T}\D{Y},\
B_4=\D{Y}\D{T},
\end{equation*}
we set
\begin{equation}\label{qwertyujh}
\begin{gathered}
\Delta_{n,z}(u+cv,x)=\dfrac{(u+cv)(x/c)\E{Y}-[(u+cv)(1+x)-z]\E{T}}{\sqrt{B_1n}},
\\
\Lambda_{n,z}(u+cv,x)
=\dfrac{B_1n-\big(B_2[(u+cv)(1+x)-z]+B_3(u+cv)x/c\big)}{\sqrt{B_1B_4n}}.
\end{gathered}
\end{equation}

By Lemma~\ref{erwtjytk}, we have the identity
\begin{equation*}
\mathcal{Y}^2_{n,z}(u+cv,x)+\mathcal{T}^2_n(u+cv,x)=\Lambda_{n,z}^2(u+cv,x)+\Delta^2_{n,z}(u+cv,x),
\end{equation*}
and equation \eqref{wertyhjty} rewrites as
\begin{multline}\label{dtfyjutykm}
\MainApprox{t}(u,c\mid\T{1}=v)=\frac{u+cv}{2\pi
c\sqrt{\D{T}\D{Y}}}\sum_{n=\EnOne}^{\infty}n^{-1}\int_{0}^{\frac{c(t-v)}{u+cv}}
\frac{1}{1+x}\int_{0}^{(u+cv)(1+x)}\P\left\{Y>z\right\}
\\
\times\exp\big\{-\tfrac{1}{2}\big[\Lambda_{n,z}^2(u+cv,x)+\Delta^2_{n,z}(u+cv,x)\big]\big\}dxdz.
\end{multline}
Equation \eqref{asdfghd} rewrites as
\begin{multline}\label{ertgerh}
\MainRemTerm{t}(u,c\mid\T{1}=v)=K(u+cv)\sum_{n=\EnOne}^{\infty}n^{-3/2}
\int_{0}^{\frac{c(t-v)}{u+cv}}\frac{1}{1+x}
\int_{0}^{(u+cv)(1+x)}\P\left\{Y>z\right\}\,
\\
\times\big(1+\big[\Lambda_{n,z}^2(u+cv,x)+\Delta^2_{n,z}(u+cv,x)\big]^{1/2}\big)^{-3}dxdz.
\end{multline}

%%%%%%%%%%%%%%%%%%%%%%%%%%%%%%%%%%%%%%%%%%%%%%%%%%%%%%%%%%%%%%%%%%%%%%%%%%%%%%%%%%%%%%%%
\subsection*{Elimination of terms with $z$ in (\ref{dtfyjutykm})}\label{wergrhetyjntr}
%%%%%%%%%%%%%%%%%%%%%%%%%%%%%%%%%%%%%%%%%%%%%%%%%%%%%%%%%%%%%%%%%%%%%%%%%%%%%%%%%%%%%%%%

Written in terms of elementary functions and considered as functions of $z$,
the expressions \eqref{dtfyjutykm} and \eqref{ertgerh} are liable to such
standard calculus manipulations as, e.g., use of Taylor's formula.

Let us write
\begin{multline}\label{ewdtehfdj}
\SeqApprox{t}{1}(u,c\mid\T{1}=v)=\frac{\E{Y}(u+cv)}{2\pi
c\sqrt{\D{T}\D{Y}}}\sum_{n=\EnOne}^{\infty}n^{-1}\int_{0}^{\frac{c(t-v)}{u+cv}}
\frac{1}{1+x}
\\
\times\exp\big\{-\tfrac{1}{2}\big[\Lambda^2_{n}(u+cv,x)+\Delta_{n}^2(u+cv,x)\big]\big\}dx,
\end{multline}
where (see \eqref{qwertyujh})
\begin{equation}\label{wq3erewg}
\Lambda_{n}(u+cv,x)=\Lambda_{n,0}(u+cv,x),\quad
\Delta_{n}(u+cv,x)=\Delta_{n,0}(u+cv,x).
\end{equation}
We need to show that
\begin{equation*}
\sup_{t>0}\,\Big|\,\MainApprox{t}(u,c\mid\T{1}=v)-\SeqApprox{t}{1}(u,c\mid\T{1}=v)\,\Big|
=\underline{O}\bigg(\frac{\ln(u+cv)}{u+cv}\bigg),
\end{equation*}
as $u+cv\to\infty$.

Writing
$f(z)=\exp\big\{-\tfrac{1}{2}\big[\Delta^2_{n,z}(u+cv,x)+\Lambda_{n,z}^2(u+cv,x)\big]\big\}>0$
and bearing in mind that $\int_{0}^{\infty}\P\left\{Y>z\right\}dz=\E{Y}$,
we divide the region of integration with respect to $z$ in \eqref{dtfyjutykm}
in two parts, $[0,\Uex]$ and $[\Uex,(u+cv)(1+x)]$, where
$\Uex=\epsilon(u+cv)(1+x)$, $0<\epsilon<1$.

Bearing in mind that for $z\in[\Uex,(u+cv)(1+x)]$
\begin{multline*}
\sup_{z>0}\exp\big\{-\tfrac{1}{2}\big[\Lambda_{n,z}^2(u+cv,x)+\Delta^2_{n,z}(u+cv,x)\big]\big\}
=\exp\big\{-\tfrac{1}{2}\mathcal{T}^2_n(u+cv,x)\big\}
\\[4pt]
\times\sup_{z>0}\exp\big\{-\tfrac{1}{2}\mathcal{Y}^2_{n,z}(u+cv,x)\big\}
\leqslant K\exp\big\{-\tfrac{1}{2}\mathcal{T}^2_n(u+cv,x)\big\}
\end{multline*}
and using Chebyshev's inequality $\P\left\{Y>z\right\}\leqslant\E{Y^3}/z^3$ which yields
\begin{equation*}
\int_{\Uex}^{(u+cv)(1+x)}\P\left\{Y>z\right\}dz\leqslant\E{Y^3}\int_{\Uex}^{(u+cv)(1+x)}\frac{dz}{z^3}
\leqslant\frac{K}{(u+cv)^2(1+x)^2},
\end{equation*}
we have
\begin{multline*}
\frac{u+cv}{2\pi c\sqrt{\D{T}\D{Y}}}\sum_{n=\EnOne}^{\infty}n^{-1}\int_{0}^{\frac{c(t-v)}{u+cv}}
\frac{1}{1+x}\int_{\Uex}^{(u+cv)(1+x)}\P\left\{Y>z\right\}
\\[4pt]
\times\exp\big\{-\tfrac{1}{2}\big[\Lambda_{n,z}^2(u+cv,x)+\Delta^2_{n,z}(u+cv,x)\big]\big\}dxdz
\\[4pt]
\leqslant\frac{K}{u+cv}\int_{0}^{\frac{c(t-v)}{u+cv}}
\frac{1}{(1+x)^3}\sum_{n=\EnOne}^{\infty}n^{-1}\exp\big\{-\tfrac{1}{2}\mathcal{T}^2_n(u+cv,x)\big\}dx
=\underline{O}((u+cv)^{-1}),
\end{multline*}
as $u+cv\to\infty$, which is checked by easy calculus.

For $z\in[0,\Uex]$, bearing in mind that
$\int_{0}^{\Uex}z\P\left\{Y>z\right\}dz\leqslant\E{Y^2}/2<\infty$ and using
Taylor's formula
\begin{equation*}
|f(z)-f(0)|\leqslant z\sup_{\xi\in[0,\Uex]}|f^{\prime}(\xi)|,
\end{equation*}
where\footnote{For brevity, here and in the sequel, we omit the arguments of
$\Delta_{n,z}(u+cv,x)$ and $\Lambda_{n,z}(u+cv,x)$ whenever it does not create
confusion.}
$f^{\prime}(z)=f(z)(\D{Y}B_1n)^{-1/2}(\E{T}{\sqrt{\D{Y}}}\Delta_{n,z}
+\E{Y}\sqrt{\D{T}}\Lambda_{n,z})$, we have to check that
\begin{multline*}
\frac{u+cv}{2\pi
c\sqrt{\D{T}\D{Y}}}\sum_{n=\EnOne}^{\infty}n^{-1}(\D{Y}B_1n)^{-1/2}\int_{0}^{\frac{c(t-v)}{u+cv}}
\frac{1}{1+x}\int_{0}^{\Uex}z\P\left\{Y>z\right\}dz
\\
\times\sup_{\xi\in[0,\Uex]}\big(f(\xi)\,\big|\big(\E{T}{\sqrt{\D{Y}}}\Delta_{n,\xi}
+\E{Y}\sqrt{\D{T}}\Lambda_{n,\xi}\big)\big|\big)dx=\underline{O}\bigg(\frac{\ln(u+cv)}{u+cv}\bigg),
\end{multline*}
as $u+cv\to\infty$. It follows from
\begin{multline}\label{drgterhr}
(u+cv)\sum_{n=\EnOne}^{\infty}n^{-3/2}\int_{0}^{\frac{c(t-v)}{u+cv}}
\frac{1}{1+x}\sup_{\xi\in[0,\Uex]}\big(f(\xi)\;\big|\Lambda_{n,\xi}\big|\big)\,dx\leqslant K(u+cv)\sum_{n=\EnOne}^{\infty}n^{-3/2}
\\
\times\int_{0}^{\frac{c(t-v)}{u+cv}} \frac{1}{1+x}|\Lambda_{n}|
\exp\big\{-\tfrac{1}{2}\big[\Lambda_{n}^2+\Delta^2_{n}\big]\big\}\,dx
=\underline{O}\bigg(\frac{\ln(u+cv)}{u+cv}\bigg)
\end{multline}
and
\begin{multline}\label{wertgr3eyh}
(u+cv)\sum_{n=\EnOne}^{\infty}n^{-3/2}\int_{0}^{\frac{c(t-v)}{u+cv}}
\frac{1}{1+x}\sup_{\xi\in[0,\Uex]}\big(f(\xi)\;\big|\Delta_{n,\xi}\big|\big)\,dx
\leqslant K(u+cv)\sum_{n=\EnOne}^{\infty}n^{-3/2}
\\
\times\int_{0}^{\frac{c(t-v)}{u+cv}} \frac{1}{1+x}|\Delta_{n}|
\exp\big\{-\tfrac{1}{2}\big[\Lambda_{n}^2+\Delta^2_{n}\big]\big\}\,dx
=\underline{O}\bigg(\frac{\ln(u+cv)}{u+cv}\bigg),
\end{multline}
as $u+cv\to\infty$. The proof of \eqref{drgterhr} and \eqref{wertgr3eyh} by
means of core asymptotic analysis of the expressions of the second kind is
deferred until Step~4.

%%%%%%%%%%%%%%%%%%%%%%%%%%%%%%%%%%%%%%%%%%%%%%%%%%%%%%%%%%%%%%%%%%%%%%%%%%%%%%%%%%%%%%%%
\subsection*{Elimination of terms with $z$ in (\ref{asdfghd})}\label{qwertyhrjty}
%%%%%%%%%%%%%%%%%%%%%%%%%%%%%%%%%%%%%%%%%%%%%%%%%%%%%%%%%%%%%%%%%%%%%%%%%%%%%%%%%%%%%%%%

Similarly to what just has been done, for
\begin{multline}\label{q345t43y346y}
\SeqRemTerm{t}{1}(u,c\mid\T{1}=v)=K(u+cv)\sum_{n=\EnOne}^{\infty}n^{-3/2}
\int_{0}^{\frac{c(t-v)}{u+cv}}\frac{1}{1+x}
\\
\times\big(1+\big[\Lambda^2_{n}(u+cv,x)+\Delta_{n}^2(u+cv,x)\big]^{1/2}\big)^{-3}dx
\end{multline}
we need to show that
\begin{equation}\label{edfghjty}
\sup_{t>0}\,\Big|\,\MainRemTerm{t}(u,c\mid\T{1}=v)-\SeqRemTerm{t}{1}(u,c\mid\T{1}=v)\,\Big|
=\underline{O}\bigg(\frac{\ln(u+cv)}{u+cv}\bigg),
\end{equation}
as $u+cv\to\infty$.

We divide as above the region of integration with respect to $z$ in
\eqref{ertgerh} in two parts, $[0,\Uex]$ and $[\Uex,(u+cv)(1+x)]$, where
$\Uex=\epsilon(u+cv)(1+x)$, $0<\epsilon<1$. On the latter, we use Chebyshev's
inequality $\P\left\{Y>z\right\}\leqslant\E{Y^3}/z^3$. On the former, bearing
in mind that $\int_{0}^{\Uex}z\P\left\{Y>z\right\}dz\leqslant\E{Y^2}/2$, we put
$g(z)=\big(1+\big[\Delta^2_{n,z}(u+cv,x)+\Lambda_{n,z}^2(u+cv,x)\big]^{1/2}\big)^{-3}$
and use Taylor's formula
\begin{equation*}
\mid g(z)-g(0)\mid\,\leqslant z\sup_{\xi\in[0,\Uex]}|g^{\prime}(\xi)|,
\end{equation*}
where $g^{\prime}(z)=-\frac{3}{\sqrt{B_1n}}\big(\E{T}\Delta_{n,z}
+\frac{B_2}{\sqrt{B_4}}\Lambda_{n,z}\big)(\Delta^2_{n,z}+\Lambda^2_{n,z})^{-1/2}g^{4/3}(z)$.
The proof reduces to checking that for all $t>0$
\begin{multline}\label{asergtshd}
(u+cv)\sum_{n=\EnOne}^{\infty}n^{-3/2}
\int_{0}^{\frac{c(t-v)}{u+cv}}\frac{1}{1+x}\int_{\Uex}^{(u+cv)(1+x)}z^{-3}dz\,
\big(1+\big[\mathcal{T}^2_n(u+cv,x)\big]^{1/2}\big)^{-3}dx
\\
=\underline{O}((u+cv)^{-1})
\end{multline}
and
\begin{multline}\label{asergthn}
(u+cv)\sum_{n=\EnOne}^{\infty}n^{-2}
\int_{0}^{\frac{c(t-v)}{u+cv}}\frac{1}{1+x}\big(K_1|\Lambda_{n}(u+cv,x)|+K_2|\Delta_{n}(u+cv,x)|\big)
\\
\times\big(1+\big[\Lambda^2_{n}(u+cv,x)+\Delta_{n}^2(u+cv,x)\big]^{1/2}\big)^{-3}dx
=\underline{O}\bigg(\frac{\ln(u+cv)}{u+cv}\bigg),
\end{multline}
as $u+cv\to\infty$. The proof of \eqref{asergtshd} and \eqref{asergthn} by
means of core asymptotic analysis of the expressions of the first kind, is
deferred until Step~4.

%%%%%%%%%%%%%%%%%%%%%%%%%%%%%%%%%%%%%%%%%%%%%%%%%%%%%%%%%%%%%%%%%%%%%%%%%%%%%%%%%%%%%%%%
\subsection*{Estimation of (\ref{q345t43y346y})}\label{456yu4e3424}
%%%%%%%%%%%%%%%%%%%%%%%%%%%%%%%%%%%%%%%%%%%%%%%%%%%%%%%%%%%%%%%%%%%%%%%%%%%%%%%%%%%%%%%%

We need to show that
\begin{equation}\label{e5y34y56uy4}
\sup_{t>0}\,\SeqRemTerm{t}{1}(u,c\mid\T{1}=v)=\underline{O}\bigg(\frac{\ln(u+cv)}{u+cv}\bigg),
\end{equation}
as $u+cv\to\infty$. The proof of \eqref{e5y34y56uy4} with
$\SeqRemTerm{t}{1}(u,c\mid\T{1}=v)$ written down in \eqref{q345t43y346y},
carried out by means of core asymptotic analysis of the expressions of the
first kind, is deferred until Step~4.

%%%%%%%%%%%%%%%%%%%%%%%%%%%%%%%%%%%%%%%%%%%%%%%%%%%%%%%%%%%%%%%%%%%%%%%%%%%%%%%%%%%%%%%%
\subsection{Step 4: core asymptotic analysis}\label{sdfdeghr}
%%%%%%%%%%%%%%%%%%%%%%%%%%%%%%%%%%%%%%%%%%%%%%%%%%%%%%%%%%%%%%%%%%%%%%%%%%%%%%%%%%%%%%%%

Before we formulate and prove the main results of this section, we examine in
more detail $\Lambda_{n}(u+cv,x)$ and $\Delta_{n}(u+cv,x)$ defined in
\eqref{wq3erewg}. From the definition, it follows straightforwardly that
\begin{align}
\Lambda_{n}(u+cv,x)&=\dfrac{B_1n-\big(B_2(u+cv)(1+x)+B_3(u+cv)x/c\big)}{\sqrt{B_1B_4n}}\notag
\\
&=\dfrac{B_1n+B_3(u+cv)/c-(1+x)(u+cv)(B_2+B_3/c)}{\sqrt{B_1B_4n}},\label{weqrtehje}
\\
\Delta_{n}(u+cv,x)&=\dfrac{((u+cv)x/c)\E{Y}-(u+cv)(1+x)\E{T}}{\sqrt{B_1n}}\notag
\\
&=\dfrac{(u+cv)(1+x)(\E{Y}/c-\E{T})-(u+cv)\E{Y}/c}{\sqrt{B_1n}},\label{qwerthrj}
\end{align}
and that for\footnote{Since we are concerned with uniform bounds, we are ready
to stretch out the range $0<x<\frac{c(t-v)}{u+cv}$ in \eqref{dtfyjutykm} and
\eqref{ertgerh} up to $0<x<\infty$.} $0<x<\infty$
\begin{equation*}
-\infty<\Lambda_{n}(u+cv,x)\leqslant\frac{B_1n-(u+cv)B_2}{\sqrt{B_1B_4n}}
=\frac{\sqrt{B_1}}{\sqrt{B_4}}\Big(\sqrt{n}-\frac{B_2}{B_1}\frac{(u+cv)}{\sqrt{n}}\Big).
\end{equation*}

In Lemmas~\ref{erwtjytk}--\ref{asdfhjnmkgh}, we proved a
number of identities for $\Delta_{n}(\OneVar,\TwoVar)$,
$\Lambda_{n}(\OneVar,\TwoVar)$ defined in \eqref{wqerertjr4}. In particular, they hold for
$\Delta_{n}(u+cv,x)=\Delta_{n}(\OneVar,\TwoVar)\mid_{\OneVar=(u+cv)(1+x),
\TwoVar=(u+cv)x/c}$ and
$\Lambda_{n}(u+cv,x)=\Lambda_{n}(\OneVar,\TwoVar)\mid_{\OneVar=(u+cv)(1+x),
\TwoVar=(u+cv)x/c}$. Let us establish two more identities for
$\Delta_{n}(u+cv,x)$ and $\Lambda_{n}(u+cv,x)$.

\begin{lemma}\label{ewrgfwdgsgw}
The following identities hold true:
\begin{multline}\label{wertg53hy654}
\Delta_{n}(u+cv,x)=\frac{\sqrt{B_4}(\E{Y}/c-\E{T})}{(B_2+B_3/c)}
\bigg[-\Lambda_{n}(u+cv,x)
\\
+\frac{\sqrt{B_1}}{\sqrt{B_4}}\bigg(\sqrt{n}+\frac{B_3}{B_1c}\frac{(u+cv)}{\sqrt{n}}\bigg)\bigg]
-\frac{\E{Y}}{c\sqrt{B_1}}\frac{(u+cv)}{\sqrt{n}}
\end{multline}
and
\begin{multline}\label{ety45tuj65}
1+x=\frac{\E{Y}\sqrt{B_1B_4}}{(B_1+B_3(\E{Y}/c-\E{T}))}\frac{\sqrt{n}}{(u+cv)}\bigg[-\Lambda_{n}(u+cv,x)
\\
+\frac{\sqrt{B_1}}{\sqrt{B_4}}\bigg(\sqrt{n}+\frac{B_3}{B_1c}\frac{(u+cv)}{\sqrt{n}}\bigg)\bigg].
\end{multline}
\end{lemma}

\begin{proof}
From \eqref{qwerthrj} and \eqref{weqrtehje}, we have two expressions for
$(1+x)$:
\begin{equation}\label{wertgrehrt}
\begin{aligned}
1+x&=\frac{\sqrt{B_1n}\Delta_{n}(u+cv,x)+(u+cv)\E{Y}/c}{(u+cv)(\E{Y}/c-\E{T})},
\\[4pt]
1+x&=\frac{B_1n+B_3(u+cv)/c-\sqrt{B_1B_4n}\Lambda_{n}(u+cv,x)}{(u+cv)(B_2+B_3/c)}.
\end{aligned}
\end{equation}
To get \eqref{wertg53hy654}, we equate the right-hand sides of both equations
\eqref{wertgrehrt} and do straightforward algebraic transformations. To have
\eqref{ety45tuj65}, we transform the right-hand side of the second equation
\eqref{wertgrehrt}, bearing in mind that $B_1=\E{Y}B_2+\E{T}B_3$ and
consequently that $B_2+B_3/c=(B_1+B_3(\E{Y}/c-\E{T}))/\E{Y}$.
\end{proof}

%%%%%%%%%%%%%%%%%%%%%%%%%%%%%%%%%%%%%%%%%%%%%%%%%%%%%%%%%%%%%%%%%%%%%%%%%%%%%%%%%%%%%%%%
\subsection*{Asymptotic analysis of the expressions of the first kind}\label{sdfgjhrghm}
%%%%%%%%%%%%%%%%%%%%%%%%%%%%%%%%%%%%%%%%%%%%%%%%%%%%%%%%%%%%%%%%%%%%%%%%%%%%%%%%%%%%%%%%

By the expressions of the first kind we call those arising in the analysis of
the remainder term in the approximation \eqref{werkyhjrtjr}. Their integrands
contain rational functions modified by a square root. The first expression of
this type (cf. \eqref{q345t43y346y} and \eqref{e5y34y56uy4}) is
\begin{equation*}
\IntOnE=(u+cv)\sum_{n=\EnOne}^{\infty}n^{-3/2}\int_{0}^{\infty}\frac{1}{1+x}
\big(1+\big[\Lambda^2_{n}(u+cv,x)+\Delta_{n}^2(u+cv,x)\big]^{1/2}\big)^{-3}dx.
\end{equation*}
Other expressions of this type are (cf. \eqref{asergthn})
\begin{multline*}
\IntOne{1}=(u+cv)\sum_{n=\EnOne}^{\infty}n^{-2}\int_{0}^{\infty}\frac{1}{1+x}|\Lambda_{n}(u+cv,x)|
(1+\big[\Lambda^2_{n}(u+cv,x)
\\
+\Delta_{n}^2(u+cv,x)\big]^{1/2}\big)^{-3}dx
\end{multline*}
and
\begin{multline*}
\IntOne{2}=(u+cv)\sum_{n=\EnOne}^{\infty}n^{-2}\int_{0}^{\infty}\frac{1}{1+x}|\Delta_{n}(u+cv,x)|
(1+\big[\Lambda^2_{n}(u+cv,x)
\\
+\Delta_{n}^2(u+cv,x)\big]^{1/2})^{-3}dx.
\end{multline*}

%%%%%%%%%%%%%%%%%%%%%%%%%%%%%%%%%%%%%%%%%%%%%%%%%%%%%%%%%%%%%%%%%%%%%%%%%%%%%%%%%%%%%%%%
\subsection*{\textit{Processing of}\hskip 6pt $\IntOnE$}\label{ertyhetjr}
%%%%%%%%%%%%%%%%%%%%%%%%%%%%%%%%%%%%%%%%%%%%%%%%%%%%%%%%%%%%%%%%%%%%%%%%%%%%%%%%%%%%%%%%

Applying both identities of Lemma~\ref{ewrgfwdgsgw}, we rewrite it as
\begin{equation*}
\begin{aligned}
\IntOnE&=\frac{(B_1+B_3(\E{Y}/c-\E{T}))}{\E{Y}\sqrt{B_1B_4}}(u+cv)^2
\sum_{n=\EnOne}^{\infty}n^{-1}\int_{0}^{\infty}\bigg\{-\Lambda_{n}(u+cv,x)
\\[4pt]
&+\frac{\sqrt{B_1}}{\sqrt{B_4}}\bigg(\sqrt{n}
+\frac{B_3}{B_1c}\frac{(u+cv)}{\sqrt{n}}\bigg)\bigg\}^{-1}
\Bigg(1+\bigg[\Lambda^2_{n}(u+cv,x)
+\bigg\{\frac{\sqrt{B_4}(\E{Y}/c-\E{T})}{B_2+B_3/c}
\\[4pt]
&\times\bigg[-\Lambda_{n}(u+cv,x)+\frac{\sqrt{B_1}}{\sqrt{B_4}}\bigg(\sqrt{n}
+\frac{B_3}{B_1c}\frac{(u+cv)}{\sqrt{n}}\bigg)\bigg]
-\frac{\E{Y}}{c\sqrt{B_1}}\frac{(u+cv)}{\sqrt{n}}\bigg\}^2\bigg]^{1/2}\Bigg)^{-3}dx.
\end{aligned}
\end{equation*}

Making the change of variables $\xi=-\Lambda_{n}(u+cv,x)$ in the integral with
respect to $x$ and bearing in mind that
\begin{equation*}
dx=\frac{\sqrt{B_1B_4n}}{(u+cv)(B_2+B_3/c)}\,d\xi,
\end{equation*}
we get
\begin{multline}\label{sdfghjfm}
\IntOnE=(u+cv)\sum_{n=\EnOne}^{\infty}n^{-3/2}\int_{L_{u+cv,n}}^{\infty}\big(\xi+R_{u+cv,n}\big)^{-1}
\\[-2pt]
\times\big(1+\big[\xi^2+\big(\cK{}\,\big(\xi+R_{u+cv,n}\big)-M_{u+cv,n}\big)^2\big]^{1/2}\big)^{-3}\,d\xi,
\end{multline}
where\footnote{Bear in mind that $cB_1+B_3[\E{Y}-c\E{T}]=\E{Y}(c\E{Y}\D{T}+\E{T}\D{Y})>0$.}
\begin{equation*}
\begin{gathered}
\cK{}=\dfrac{\E{Y}\sqrt{B_4}[\E{Y}-c\E{T}]}{cB_1+B_3[\E{Y}-c\E{T}]}\,
\begin{cases}
\,>0,&c<\frac{\E{Y}}{\E{T}},
\\[2pt]
\,=0,&c=\frac{\E{Y}}{\E{T}},
\\[2pt]
\,<0,&c>\frac{\E{Y}}{\E{T}},
\end{cases}
\quad
M_{u+cv,n}=\frac{\E{Y}}{c\sqrt{B_1}}\frac{u+cv}{\sqrt{n}}>0,\
\\[8pt]
L_{u+cv,n}=\frac{\sqrt{B_1}}{\sqrt{B_4}}
\bigg(\frac{B_2}{B_1}\frac{u+cv}{\sqrt{n}}-\sqrt{n}\bigg),
\quad
R_{u+cv,n}=\frac{\sqrt{B_1}}{\sqrt{B_4}}
\bigg(\frac{B_3}{B_1c}\frac{u+cv}{\sqrt{n}}+\sqrt{n}\bigg)>0,
\end{gathered}
\end{equation*}
and
\begin{equation*}
M_{u+cv,n}-\cK{}R_{u+cv,n}
=\frac{\E{Y}\sqrt{B_1}}{cB_1+B_3[\E{Y}-c\E{T}]}\left(\frac{u+cv}{\sqrt{n}}-[\E{Y}-c\E{T}]\sqrt{n}\right).
\end{equation*}
In the sequel, put for brevity $\cS=\E{Y}/\E{T}$.

\begin{lemma}\label{ewrehgwew}
We have $\IntOnE=\underline{O}\bigg(\dfrac{\ln(u+cv)}{u+cv}\bigg)$, as
$u+cv\to\infty$.
\end{lemma}

First, we prove Lemma~\ref{ewrehgwew} for $c=\cS$, bearing in mind that $K_{\cS}=0$. Then we prove it
for $c\ne\cS$.

\begin{proof}[Proof of Lemma~\ref{ewrehgwew} for $c=\cS$]
Let us put for brevity $U=u+\cS v$ and $\hat{L}=\hatL$, $\hat{R}=\hatR$,
$\hat{M}=\hatM$, i.e., for
$\mcA=\frac{\E{T}}{\sqrt{B_1}}\frac{\E{T}\sqrt{\D{Y}}}{\E{Y}\sqrt{\D{T}}}>0$,\
$\mcB=\frac{\sqrt{B_1}}{\sqrt{\D{Y}\D{T}}}>0$,\
$\mcC=\frac{\E{T}}{\sqrt{B_1}}>0$,\ we have
\begin{equation*}
\begin{gathered}
\hat{L}=\frac{\mcC^2}{\mcA}\frac{U}{\sqrt{n}}-\mcB\sqrt{n}\,
\begin{cases}
\,>0,\
n<\frac{B_2}{B_1}U,
\\[6pt]
\,<0,\ n>\frac{B_2}{B_1}U,
\end{cases}
\quad
\hat{R}=\mcA\frac{U}{\sqrt{n}}+\mcB\sqrt{n}>0,
\\[6pt]
0<\hat{M}=\mcC\dfrac{U}{\sqrt{n}}\,
\begin{cases}
\,>1,\
n<\frac{(\E{T})^2}{B_1}U^2,
\\[6pt]
\,<1,\ n>\frac{(\E{T})^2}{B_1}U^2,
\end{cases}
\quad \hat{L}+\hat{R}=\frac{\mcC^2+\mcA^2}{\mcA}\frac{U}{\sqrt{n}}>0
\end{gathered}
\end{equation*}
and \hskip 2pt $\IntOneHaT=\IntOnE\big|_{c=\cS}$, i.e.,
$\IntOneHaT=U\sum_{n>\EnOne}n^{-3/2}\int_{\hat{L}}^{\infty}\big(\xi+\hat{R}\big)^{-1}
\big(1+\big[\xi^2+\hat{M}^2\big]^{1/2}\big)^{-3}\,d\xi$.

\smallskip
It is easily seen that
\begin{equation}\label{ewrgteyherj}
\IntOneHaT\leqslant K_1\Integral{1}+K_2\Integral{2},
\end{equation}
where
\begin{equation*}
\begin{aligned}
\Integral{1}&=U\sum_{\EnOne<n<\frac{(\E{T})^2}{B_1}\,U^2}
n^{-3/2}\int_{\hat{L}}^{\infty}\big(\xi+\hat{R}\big)^{-1}
\big(\xi^2+\hat{M}^2\big)^{-3/2}\,d\xi,
\\[-2pt]
\Integral{2}&=U\sum_{n>\frac{(\E{T})^2}{B_1}\,U^2}
n^{-3/2}\int_{\hat{L}}^{\infty}\big(\xi+\hat{R}\big)^{-1}
\big(1+(2\hat{M}|\xi|)^{1/2}\big)^{-3}\,d\xi.
\end{aligned}
\end{equation*}

The essence of \eqref{ewrgteyherj} is the following. For $n$ such that
$\hat{M}>1$, we simplify the denominator $(1+[\xi^2+\hat{M}^2]^{1/2})^{3}$ by
switching to $(\xi^2+\hat{M}^2)^{3/2}$. The latter has no singularity since
$\hat{M}>1$. For $n$ such that $\hat{M}<1$, we keep $1$ in the denominator and
use the inequality between the arithmetic mean and the geometric mean. Both
these estimates are such that the integrals in $\Integral{1}$ and
$\Integral{2}$ may be evaluated explicitly.

Examining $\Integral{1}$, the explicit expression for the integral
$\int_{\hat{L}}^{\infty}(\xi+\hat{R})^{-1}(\xi^2+\hat{M}^2)^{-3/2}\,d\xi$ is
found in Lemma~\ref{rthfgnmhfgj}. Using it, the asymptotic behavior of
$\Integral{1}$, as $U\to\infty$, is checked as required in
Section~\ref{qasrdtjkyl}.

Examining $\Integral{2}$ and bearing in mind that $\hat{L}<0$ for
$n>\frac{(\E{T})^2}{B_1}\,U^2$, we split the integrand and make the change of
variables as follows:
\begin{multline*}
\int_{\hat{L}}^{\infty}\big(\xi+\hat{R}\big)^{-1}
\big(1+(2\hat{M}|\xi|)^{1/2}\big)^{-3}\,d\xi
\\
=\int_{0}^{\infty}\big(\xi+\hat{R}\big)^{-1}
\big(1+(2\hat{M}\xi)^{1/2}\big)^{-3}\,d\xi+\int_{0}^{-\hat{L}}\big(\hat{R}-\xi\big)^{-1}
\big(1+(2\hat{M}\xi)^{1/2}\big)^{-3}\,d\xi
\\
=\int_{0}^{\infty}\big(2\hat{M}\hat{R}+\zeta\big)^{-1}
\big(1+\sqrt{\zeta}\big)^{-3}\,d\zeta+\int_{0}^{-\frac{\hat{L}}{2\hat{M}}}\big(2\hat{M}\hat{R}-\zeta\big)^{-1}
\big(1+\sqrt{\zeta}\big)^{-3}\,d\zeta.
\end{multline*}
The explicit expressions for two latter integrals are found in Lemma~\ref{wqerthrjnmr}.
Using this result, the asymptotic behavior of $\Integral{2}$, as $U\to\infty$, is checked as required
in Section~\ref{qwerjkgulgh}. The proof is complete.
\end{proof}

\begin{proof}[Proof of Lemma~\ref{ewrehgwew} for $c\ne\cS$]
Let us put for brevity $U=u+cv$ and $\tilde{L}=\XatL$, $\tilde{R}=\XatR$,
$\tilde{M}=\XatM$, i.e., for\footnote{If we put $c=\cS$ in these expressions,
they will be equal to $\mcA$, $\mcB$, $\mcC$ introduced in the proof of
Lemma~\ref{ewrehgwew} for $c=\cS$.} $\mcA=\frac{\E{T}}{\sqrt{B_1}}
\frac{\sqrt{\D{Y}}}{c\sqrt{\D{T}}}>0$,\
$\mcB=\frac{\sqrt{B_1}}{\sqrt{\D{Y}\D{T}}}>0$,\
$\mcC=\frac{\E{Y}}{c\sqrt{B_1}}>0$, we have
\begin{equation}\label{sdfghffghjfgj}
\begin{gathered}
\tilde{L}=\frac{c\,\E{T}\mcC^2}{\E{Y}\mcA}\frac{U}{\sqrt{n}}-\mcB\sqrt{n}\,
\begin{cases}
\,>0,\ n<\frac{B_2}{B_1}U,
\\[6pt]
\,<0,\ n>\frac{B_2}{B_1}U,
\end{cases}
\quad
\tilde{R}=\mcA\frac{U}{\sqrt{n}}+\mcB\sqrt{n}>0,\quad
\\[8pt]
0<\tilde{M}=\mcC\frac{U}{\sqrt{n}}\,
\begin{cases}
\,>1,\
n<\frac{(\E{Y})^2}{c^2B_1}U^2,
\\[6pt]
\,<1,\ n>\frac{(\E{Y})^2}{c^2B_1}U^2,
\end{cases}
\quad
\tilde{L}+\tilde{R}=\frac{c\E{T}\mcC^2+\E{Y}\mcA^2}{\E{Y}\mcA}\frac{U}{\sqrt{n}}>0.
\end{gathered}
\end{equation}
Bearing in mind that
$cB_1+B_3[\E{Y}-c\E{T}]=\E{Y}(c\E{Y}\D{T}+\E{T}\D{Y})>0$, we have
\begin{equation}\label{wqqgbdf}
\begin{aligned}
\tilde{M}-\cK{}\tilde{R}&=\frac{\E{Y}\sqrt{B_1}}{cB_1+B_3[\E{Y}-c\E{T}]}
\,\bigg(\frac{U}{\sqrt{n}}-[\E{Y}-c\E{T}]\sqrt{n}\bigg)
\\[4pt]
&=\big(\mcC-\cK{}\mcA\big)\tfrac{U}{\sqrt{n}}-\cK{}\mcB\sqrt{n}.
\end{aligned}
\end{equation}

Let us rewrite \eqref{sdfghjfm} as
\begin{multline*}
\IntOnE=U\sum_{n=\EnOne}^{\infty}n^{-3/2}\int_{\tilde{L}}^{\infty}\big(\xi+\tilde{R}\big)^{-1}
\big(1+\big[(1+\cK{2})\xi^2-2\cK{}(\tilde{M}-\cK{}\tilde{R})\xi
\\[-6pt]
+(\tilde{M}-\cK{}\tilde{R})^2\big]^{1/2}\big)^{-3}\,d\xi
\end{multline*}
and, completing the square and making the change of variables, rewrite it as
\begin{multline}\label{sdaffjhfgmkf}
\IntOnE=\frac{U}{(1+K^2_{c})^{3/2}}\sum_{n=\EnOne}^{\infty}n^{-3/2}
\int_{\tilde{L}-\cK{}\frac{\tilde{M}-\cK{}\tilde{R}}{1+K^2_{c}}}^{\infty}
\bigg(\zeta+\frac{\tilde{R}+\cK{}\tilde{M}}{1+K^2_{c}}\bigg)^{-1}
\\[-2pt]
\times\Bigg(\frac{1}{(1+K^2_{c})^{1/2}}+ \bigg[\zeta^2
+\bigg(\frac{\tilde{M}-\cK{}\tilde{R}}{1+K^2_{c}}\bigg)^2\bigg]^{1/2}\Bigg)^{-3}\,d\zeta.
\end{multline}

%%%%%%%%%%%%%%%%%%%%%%%%%%%%%%%%%%%%%%%%%%%%%%%%%%%%%%%%%%%%%%%%%%%%%%%%%%%%%%%%%%%%%%%%
\medskip\paragraph{\it Case $c>\cS$}\label{sdrtgfjkt}
%%%%%%%%%%%%%%%%%%%%%%%%%%%%%%%%%%%%%%%%%%%%%%%%%%%%%%%%%%%%%%%%%%%%%%%%%%%%%%%%%%%%%%%%

In this case, $\cK{}<0$. The second summand in brackets in the integrand in
\eqref{sdaffjhfgmkf} is positive since in this case
$\tilde{M}-\cK{}\tilde{R}>0$; it is easily seen from the second equality
\eqref{sdfghffghjfgj}. Moreover, for $\cK{}<0$ the difference
$\tilde{M}-\cK{}\tilde{R}$ increases, as $n$ increases, and exceeds $K\sqrt{U}$
for $n>\EnOne$. The integrand in \eqref{sdaffjhfgmkf} has no singularities in
the region of integration since
\begin{equation*}
\tilde{L}-\cK{}\frac{\tilde{M}-\cK{}\tilde{R}}{1+K^2_{c}}
+\frac{\tilde{R}+\cK{}\tilde{M}}{1+K^2_{c}}=\tilde{L}+\tilde{R}>0.
\end{equation*}

We use the estimate $\IntOnE\leqslant K_3\Integral{3}$, where
\begin{equation*}
\Integral{3}=U\sum_{n=\EnOne}^{\infty}n^{-3/2}
\int_{\tilde{L}-\cK{}\frac{\tilde{M}-\cK{}\tilde{R}}{1+\cK{2}}}^{\infty}
\bigg(\zeta+\frac{\tilde{R}+\cK{}\tilde{M}}{1+K^2_{c}}\bigg)^{-1}
\bigg(\zeta^2+\bigg(\frac{\tilde{M}-\cK{}\tilde{R}}{1+K^2_{c}}\bigg)^2\,\bigg)^{-3/2}\,d\zeta.
\end{equation*}

The explicit expression for the integral in $\Integral{3}$ is found in
Lemma~\ref{wqerthrjnmr}. Using it, the asymptotic behavior of $\Integral{3}$,
as $U\to\infty$, is checked as required in Section~\ref{werrtyjrtjrtrjk}. The
proof is complete.

%%%%%%%%%%%%%%%%%%%%%%%%%%%%%%%%%%%%%%%%%%%%%%%%%%%%%%%%%%%%%%%%%%%%%%%%%%%%%%%%%%%%%%%%
\medskip\paragraph{\it Case $c<\cS$}\label{ewrtyjrtk}
%%%%%%%%%%%%%%%%%%%%%%%%%%%%%%%%%%%%%%%%%%%%%%%%%%%%%%%%%%%%%%%%%%%%%%%%%%%%%%%%%%%%%%%%

In this case, $\cK{}>0$. We have (see \eqref{wqqgbdf})
\begin{equation*}
\tilde{M}-\cK{}\tilde{R}\,
\begin{cases}
\,>0,&n<\frac{U}{\E{T}(\cS-c)},
\\[4pt]
\,<0,&n>\frac{U}{\E{T}(\cS-c)}.
\end{cases}
\end{equation*}
It is easily seen that $\IntOnE\leqslant K_4\Integral{4}+K_5\Integral{5}+K_6\Integral{6}$, where
\begin{equation*}
\begin{aligned}
\Integral{4}&=U\hskip -8pt\sum_{\EnOne<n<\frac{U}{\E{Y}-c\E{T}}-\frac{K}{\cK{}\mcB}U}n^{-3/2}
\int_{\tilde{L}-\cK{}\frac{\tilde{M}-\cK{}\tilde{R}}{1+\cK{2}}}^{\infty}
\bigg(\zeta+\frac{\tilde{R}+\cK{}\tilde{M}}{1+\cK{2}}\bigg)^{-1}
\\[-4pt]
&\hskip 200pt\times\bigg(\zeta^2+\bigg(\frac{\tilde{M}-\cK{}\tilde{R}}{1+\cK{2}}\bigg)^2\bigg)^{-3/2}\,d\zeta,
\\[4pt]
\Integral{5}&=U\hskip -2pt\sum_{\frac{U}{\E{Y}-c\E{T}}-\frac{K}{\cK{}\mcB}U
<n<\frac{U}{\E{Y}-c\E{T}}+\frac{K}{\cK{}\mcB}U}n^{-3/2}
\int_{\tilde{L}-\cK{}\frac{\tilde{M}-\cK{}\tilde{R}}{1+\cK{2}}}^{\infty}
\bigg(\zeta+\frac{\tilde{R}+\cK{}\tilde{M}}{1+\cK{2}}\bigg)^{-1}
\\[-4pt]
&\hskip 200pt\times\bigg(1+\bigg(2\bigg(\frac{\tilde{M}-\cK{}\tilde{R}}{1+\cK{2}}\bigg)
\,|\zeta|\bigg)^{1/2}\bigg)^{-3}\,d\zeta,
\\[4pt]
\Integral{6}&=U\hskip -8pt\sum_{n>\frac{U}{\E{Y}-c\E{T}}+\frac{K}{\cK{}\mcB}U}n^{-3/2}
\int_{\tilde{L}-\cK{}\frac{\tilde{M}-\cK{}\tilde{R}}{1+\cK{2}}}^{\infty}
\bigg(\zeta+\frac{\tilde{R}+\cK{}\tilde{M}}{1+\cK{2}}\bigg)^{-1}
\\[-4pt]
&\hskip 200pt\times\bigg(\zeta^2
+\bigg(\frac{\tilde{M}-\cK{}\tilde{R}}{1+\cK{2}}\bigg)^2\bigg)^{-3/2}\,d\zeta.
\end{aligned}
\end{equation*}

It is easily seen that since
\begin{equation}\label{qwrethrj}
\frac{\mcC-\cK{}\mcA}{\cK{}\mcB}=\frac{1}{\E{Y}-c\E{T}}=\frac{\cS}{\E{Y}(\cS-c)},
\end{equation}
which can be verified by direct calculations, the range of summation
$\EnOne<n<\frac{U}{\E{Y}-c\E{T}}-\frac{K}{\cK{}\mcB}U$ in $\Integral{4}$ may we
written as $\EnOne<n<\frac{\mcC-\cK{}\mcA}{\cK{}\mcB}U-\frac{K}{\cK{}\mcB}U$,
the range of summation $\frac{U}{\E{Y}-c\E{T}}-\frac{K}{\cK{}\mcB}U
<n<\frac{U}{\E{Y}-c\E{T}}+\frac{K}{\cK{}\mcB}U$ in $\Integral{5}$ may we
written as $\frac{\mcC-\cK{}\mcA}{\cK{}\mcB}U-\frac{K}{\cK{}\mcB}U
<n<\frac{\mcC-\cK{}\mcA}{\cK{}\mcB}U+\frac{K}{\cK{}\mcB}U$, and the range of
summation $n>\frac{U}{\E{Y}-c\E{T}}+\frac{K}{\cK{}\mcB}U$ in $\Integral{6}$ may
we written as $n>\frac{\mcC-\cK{}\mcA}{\cK{}\mcB}U+\frac{K}{\cK{}\mcB}U$.

The explicit expressions for the integrals in $\Integral{4}$ and $\Integral{6}$
are similar to that one for the integral in $\Integral{3}$. Using it, the
asymptotic behavior of $\Integral{4}$ and $\Integral{6}$, as $U\to\infty$, is
checked as required in Section~\ref{saerdfgsfh}.

The explicit expression for the integral in $\Integral{5}$ is similar to that one for the integral
in $\Integral{2}$. Using it, the asymptotic behavior of $\Integral{5}$, as $U\to\infty$, is checked
as required in Section~\ref{fghrfhdfd}. The proof is complete.
\end{proof}

%%%%%%%%%%%%%%%%%%%%%%%%%%%%%%%%%%%%%%%%%%%%%%%%%%%%%%%%%%%%%%%%%%%%%%%%%%%%%%%%%%%%%%%%
\subsection*{\textit{Processing of}\hskip 6pt $\IntOne{1}$}\label{rdthtrjrywqe}
%%%%%%%%%%%%%%%%%%%%%%%%%%%%%%%%%%%%%%%%%%%%%%%%%%%%%%%%%%%%%%%%%%%%%%%%%%%%%%%%%%%%%%%%

The same way as for $\IntOnE$, rewrite $\IntOne{1}$ as
\begin{multline*}
\IntOne{1}=(u+cv)\sum_{n=\EnOne}^{\infty}n^{-2}\int_{L_{u+cv,n}}^{\infty}\big(\xi+R_{u+cv,n}\big)^{-1}
\big|\,\xi\,\big|
\\[-2pt]
\times\big(1+\big[\xi^2+\big(\cK{}\,\big(\xi+R_{u+cv,n}\big)-M_{u+cv,n}\big)^2\big]^{1/2}\big)^{-3}\,d\xi.
\end{multline*}

\begin{lemma}\label{dfetyjrtjr}
We have $\IntOne{1}=\underline{O}\bigg(\dfrac{\ln(u+cv)}{u+cv}\bigg)$,
as $u+cv\to\infty$.
\end{lemma}

\begin{proof}[Proof of Lemma~\ref{dfetyjrtjr} for $c=\cS$]
Retaining notation used in Lemma~\ref{ewrehgwew}, consider $\IntOneHat{1}=\IntOne{1}\big|_{c=\cS}$, i.e.,
\begin{equation*}
\IntOneHat{1}=U\sum_{n>\EnOne}n^{-2}\int_{\hat{L}}^{\infty}|\xi|\big(\xi+\hat{R}\big)^{-1}
\big(1+\big[\xi^2+\hat{M}^2\big]^{1/2}\big)^{-3}\,d\xi\leqslant
K_1\Integrall{1}+K_2\Integrall{2},
\end{equation*}
where\footnote{In $\Integrall{1}$, the first integral is with $\hat{L}>0$ and
the second with $\hat{L}<0$. In $\Integrall{2}$, the integral is with
$\hat{L}<0$ and $0<\hat{M}<1$.}
\begin{equation*}
\begin{aligned}
\Integrall{1}&=U\sum_{\EnOne<n<\frac{B_2}{B_1}U}
n^{-2}\int_{\hat{L}}^{\infty}|\xi|\big(\xi+\hat{R}\big)^{-1}
\big(\xi^2+\hat{M}^2\big)^{-3/2}\,d\xi
\\[-4pt]
&\hskip 60pt+U\sum_{\frac{B_2}{B_1}U<n<\frac{(\E{T})^2}{B_1}\,U^2}
n^{-2}\int_{\hat{L}}^{\infty}|\xi|\big(\xi+\hat{R}\big)^{-1}
\big(\xi^2+\hat{M}^2\big)^{-3/2}\,d\xi,
\\
\Integrall{2}&=U\sum_{n>\frac{(\E{T})^2}{B_1}\,U^2}
n^{-2}\int_{\hat{L}}^{\infty}|\xi|\big(\xi+\hat{R}\big)^{-1}
\big(1+(2\hat{M}|\xi|)^{1/2}\big)^{-3}\,d\xi.
\end{aligned}
\end{equation*}
The asymptotic behavior of $\Integrall{1}$, as $U\to\infty$, is checked as required in
Section~\ref{sadghjmgh}. The asymptotic behavior of $\Integrall{2}$, as $U\to\infty$,
is checked as required in Section~\ref{wertkktukt}.
\end{proof}

\begin{proof}[Proof of Lemma~\ref{dfetyjrtjr} for $c\ne\cS$]
This proof is a modification of the proof of Lemma \ref{dfetyjrtjr} for
$c=\cS$, alike the proof of Lemma~\ref{ewrehgwew} for $c\ne\cS$ was a
modification of that proof for $c=\cS$. It uses essentially the same techniques
and is left to the reader.
\end{proof}

%%%%%%%%%%%%%%%%%%%%%%%%%%%%%%%%%%%%%%%%%%%%%%%%%%%%%%%%%%%%%%%%%%%%%%%%%%%%%%%%%%%%%%%%
\subsection*{\textit{Processing of}\hskip 6pt $\IntOne{2}$}\label{asdffklg}
%%%%%%%%%%%%%%%%%%%%%%%%%%%%%%%%%%%%%%%%%%%%%%%%%%%%%%%%%%%%%%%%%%%%%%%%%%%%%%%%%%%%%%%%

Just as we did in the analysis of $\IntOnE$, rewrite $\IntOne{2}$ as
\begin{multline*}
\IntOne{2}=(u+cv)\sum_{n=\EnOne}^{\infty}n^{-2}\int_{L_{u+cv,n}}^{\infty}\big(\xi+R_{u+cv,n}\big)^{-1}
\,\big|\cK{}\,\big(\xi+R_{u+cv,n}\big)-M_{u+cv,n}\big|
\\[-2pt]
\times\big(1+\big[\xi^2+\big(\cK{}\,\big(\xi+R_{u+cv,n}\big)-M_{u+cv,n}\big)^2\big]^{1/2}\big)^{-3}\,d\xi.
\end{multline*}

\begin{lemma}\label{dfhgjkmgh}
We have $\IntOne{2}=\underline{O}\bigg(\dfrac{\ln(u+cv)}{u+cv}\bigg)$,
as $u+cv\to\infty$.
\end{lemma}

\begin{proof}
This proof goes along the same lines as the proof of Lemma~\ref{dfetyjrtjr} and
is left to the reader.
\end{proof}

%%%%%%%%%%%%%%%%%%%%%%%%%%%%%%%%%%%%%%%%%%%%%%%%%%%%%%%%%%%%%%%%%%%%%%%%%%%%%%%%%%%%%%%%
\subsection*{Asymptotic analysis of the expressions of the second kind}\label{srgserhgs}
%%%%%%%%%%%%%%%%%%%%%%%%%%%%%%%%%%%%%%%%%%%%%%%%%%%%%%%%%%%%%%%%%%%%%%%%%%%%%%%%%%%%%%%%

By the expressions of the second kind we call those arising when we simplify
the main term of approximation \eqref{werkyhjrtjr}. Their integrands contain
exponential, inherited from CLT, and rational functions. The first expression
of this type (cf. \eqref{ewdtehfdj}) is
\begin{equation*}
\IntTwO=(u+cv)\sum_{n=\EnOne}^{\infty}n^{-1}\int_{0}^{\infty}\frac{1}{1+x}
\exp\big\{-\tfrac{1}{2}\big[\Lambda_n^2(u+cv,x)+\Delta^2_{n}(u+cv,x)\big]\big\}dx.
\end{equation*}
Other expressions of this type are (cf. \eqref{drgterhr} and
\eqref{wertgr3eyh})
\begin{equation*}
\begin{aligned}
\IntTwo{1}&=(u+cv)\sum_{n=\EnOne}^{\infty}n^{-3/2}\int_{0}^{\infty}\frac{|\Lambda_{n}(u+cv,x)|}{1+x}
\exp\big\{-\tfrac{1}{2}\big[\Lambda_n^2(u+cv,x)
\\[-4pt]
&\hskip 250pt+\Delta^2_{n}(u+cv,x)\big]\big\}dx,
\\[-4pt]
\IntTwo{2}&=(u+cv)\sum_{n=\EnOne}^{\infty}n^{-3/2}\int_{0}^{\infty}\frac{|\Delta_{n}(u+cv,x)|}{1+x}
\exp\big\{-\tfrac{1}{2}\big[\Lambda_n^2(u+cv,x)
\\[-4pt]
&\hskip 250pt+\Delta^2_{n}(u+cv,x)\big]\big\}dx
\end{aligned}
\end{equation*}
and (see Section~\ref{asfdghjnm} below)
\begin{equation*}
\begin{aligned}
\IntTwo{3}&=(u+cv)^{1/2}\sum_{n=\EnOne}^{\infty}n^{-1}\int_{0}^{\infty}
\frac{|\Lambda_{n}(u+cv,x)|}{(1+x)^{3/2}}\exp\big\{-\tfrac{1}{2}\big[\Lambda_n^2(u+cv,x)
\\[-4pt]
&\hskip 250pt+\Delta^2_{n}(u+cv,x)\big]\big\}dx,
\\[-4pt]
\IntTwo{4}&=(u+cv)^{1/2}\sum_{n=\EnOne}^{\infty}n^{-1}\int_{0}^{\infty}
\frac{|\Delta_{n}(u+cv,x)|}{(1+x)^{3/2}}\exp\big\{-\tfrac{1}{2}\big[\Lambda_n^2(u+cv,x)
\\[-4pt]
&\hskip 250pt+\Delta^2_{n}(u+cv,x)\big]\big\}dx.
\end{aligned}
\end{equation*}

%%%%%%%%%%%%%%%%%%%%%%%%%%%%%%%%%%%%%%%%%%%%%%%%%%%%%%%%%%%%%%%%%%%%%%%%%%%%%%%%%%%%%%%%
\subsection*{\textit{Processing of}\hskip 6pt $\IntTwO$}\label{qwerrjftfgh}
%%%%%%%%%%%%%%%%%%%%%%%%%%%%%%%%%%%%%%%%%%%%%%%%%%%%%%%%%%%%%%%%%%%%%%%%%%%%%%%%%%%%%%%%

Applying the identities of Lemma~\ref{ewrgfwdgsgw} and making the change of
variables $\xi=-\Lambda_{n}(u+cv,x)$ in the integral with respect to $x$, we
rewrite it as
\begin{multline}\label{wqghjmtjrj}
\IntTwO=(u+cv)\,\sum_{n=\EnOne}^{\infty}n^{-3/2}\int_{\XatL}^{\infty}
(\xi+\XatR)^{-1}
\\
\times\exp\big\{-\tfrac12\big[\xi^2+\{\cK{}[\xi+\XatR]-\XatM\}^2\big]\big\}\,d\xi.
\end{multline}

\begin{lemma}\label{qqwehberhgert}
We have $\IntTwO=\underline{O}\bigg(\dfrac{\ln(u+cv)}{u+cv}\bigg)$,
as $u+cv\to\infty$.
\end{lemma}

As before, we prove first this lemma in the case $c=\cS$ and then in the case
$c\ne\cS$. In both cases, we use notation set in respective parts of the
proof of Lemma~\ref{ewrehgwew}.

\begin{proof}[Proof of Lemma~\ref{qqwehberhgert} for $c=\cS$]
Recall (see \eqref{sdfghffghjfgj}) that $\hat{L}>0$ for $n<\frac{B_2}{B_1}U$,
$\hat{L}<0$ for $n>\frac{B_2}{B_1}U$, that\footnote{Therefore, the integrand in
\eqref{3245terhr} does not contain singularities within the range of
integration. The unique point of singularity of the first factor lies to the
left of $\hat{L}$ since $-\hat{R}<\hat{L}$. The second factor is positive
everywhere.} $\hat{L}+\hat{R}>0$ for all $n$, and $\hat{M}>1$ for
$n<\frac{(\E{T})^2}{B_1}U^2$, $\hat{M}<1$ for $n>\frac{(\E{T})^2}{B_1}U^2$, and
consider $\IntTwoHat=\IntTwO\,\big|_{c=\cS}$, i.e.,
\begin{equation}\label{3245terhr}
\IntTwoHat=U\sum_{n>\EnOne}n^{-3/2}\exp\big\{-\tfrac12\hat{M}^2\big\}
\int_{\hat{L}}^{\infty}\big(\xi+\hat{R}\big)^{-1}\exp\big\{-\tfrac12\,\xi^2\big\}\,d\xi.
\end{equation}

It is easily seen that
\begin{equation}\label{swdrthyjh}
\IntTwoHat\leqslant K_1\Jntegral{1}+K_2\Jntegral{2},
\end{equation}
where\footnote{While using \eqref{ewrgteyherj} was essential, using of
\eqref{swdrthyjh} is largely for convenience: it emphasizes that $\hat{M}$ is
small for $n>\frac{(\E{T})^2}{B_1}U^2$, and the factor
$\exp\big\{-\tfrac12\hat{M}^2\big\}$ is unessential.}
\begin{equation*}
\begin{aligned}
\Jntegral{1}&=U\sum_{\EnOne<n<\frac{(\E{T})^2}{B_1}U^2}n^{-3/2}\exp\big\{-\tfrac12\hat{M}^2\big\}
\int_{\hat{L}}^{\infty}\big(\xi+\hat{R}\big)^{-1}\exp\big\{-\tfrac12\,\xi^2\big\}\,d\xi,
\\
\Jntegral{2}&=U\sum_{n>\frac{(\E{T})^2}{B_1}U^2}n^{-3/2}
\int_{\hat{L}}^{\infty}\big(\xi+\hat{R}\big)^{-1}\exp\big\{-\tfrac12\,\xi^2\big\}\,d\xi.
\end{aligned}
\end{equation*}
The asymptotic behavior of $\Jntegral{1}$, as $U\to\infty$, is checked as
required in Section~\ref{werfgwegfwe}. The asymptotic behavior of
$\Jntegral{2}$, as $U\to\infty$, is checked as required in
Section~\ref{ertherhrhr}. The proof is complete.
\end{proof}

\begin{proof}[Proof of Lemma~\ref{qqwehberhgert} for $c\ne\cS$]
As before (see \eqref{sdfghffghjfgj}), we put $U=u+cv$ and $\tilde{L}=\XatL$,
$\tilde{R}=\XatR$, $\tilde{M}=\XatM$. Rewrite \eqref{wqghjmtjrj} as
\begin{equation*}
\IntTwO=U\sum_{n=\EnOne}^{\infty}n^{-3/2}\int_{\tilde{L}}^{\infty}\big(\xi+\tilde{R}\big)^{-1}
\exp\big\{-\tfrac12\big[(1+\cK{2})\xi^2-2\cK{}(\tilde{M}-\cK{}\tilde{R})\xi
+(\tilde{M}-\cK{}\tilde{R})^2\big]\big\}\,d\xi
\end{equation*}
and, completing the square and making the change of variables, as
\begin{multline*}
\IntTwO=U\sum_{n=\EnOne}^{\infty}n^{-3/2}
\exp\bigg\{-\frac12\frac{(\tilde{M}-\cK{}\tilde{R})^2}{1+\cK{2}}\bigg\}
\int_{\tilde{L}-\cK{}\frac{\tilde{M}-\cK{}\tilde{R}}{1+\cK{2}}}^{\infty}
\bigg(\zeta+\frac{\tilde{R}+\cK{}\tilde{M}}{1+\cK{2}}\bigg)^{-1}
\\
\times
\exp\bigg\{-\frac{\sqrt{1+\cK{2}}}{2}\,\zeta^2\bigg\}\,d\zeta.
\end{multline*}
Since the exponential factor is easier to work, this expression is suitable for
its asymptotic analysis without its simplifying\footnote{Recall that dealing
with the analogue formula for $\IntOnE$ (see \eqref{sdaffjhfgmkf}), due to
technical complexities, we had to switch to certain upper bounds for
$\IntOnE$.}.

%%%%%%%%%%%%%%%%%%%%%%%%%%%%%%%%%%%%%%%%%%%%%%%%%%%%%%%%%%%%%%%%%%%%%%%%%%%%%%%%%%%%%%%%
\medskip\paragraph{\it Case $c>\cS$}\label{sdtfyjuykm}
%%%%%%%%%%%%%%%%%%%%%%%%%%%%%%%%%%%%%%%%%%%%%%%%%%%%%%%%%%%%%%%%%%%%%%%%%%%%%%%%%%%%%%%%

In this case, $\cK{}<0$. Recall that it yields $\tilde{M}-\cK{}\tilde{R}>0$ and
use the arguments outlined in the respective part of the proof of
Lemma~\ref{ewrehgwew}, for $c>\cS$. The asymptotic behavior of the integral
\begin{equation*}
\int_{\tilde{L}-\cK{}\frac{\tilde{M}-\cK{}\tilde{R}}{1+\cK{2}}}^{\infty}
\bigg(\zeta+\frac{\tilde{R}+\cK{}\tilde{M}}{1+\cK{2}}\bigg)^{-1}
\exp\bigg\{-\frac{\sqrt{1+\cK{2}}}{2}\,\zeta^2\bigg\}\,d\zeta
\end{equation*}
is examined by means of a direct extension, as it was done in
Section~\ref{werrtyjrtjrtrjk}, of Lemma~\ref{wertejht}. Using it, the
asymptotic behavior of $\IntTwO$, as $U\to\infty$, is easily checked as
required along the lines traced in Sections~\ref{werrtyjrtjrtrjk},
\ref{werfgwegfwe} and \ref{ertherhrhr}. The proof is complete.

%%%%%%%%%%%%%%%%%%%%%%%%%%%%%%%%%%%%%%%%%%%%%%%%%%%%%%%%%%%%%%%%%%%%%%%%%%%%%%%%%%%%%%%%
\medskip\paragraph{\it Case $c<\cS$}\label{sfdghjtsdrf}
%%%%%%%%%%%%%%%%%%%%%%%%%%%%%%%%%%%%%%%%%%%%%%%%%%%%%%%%%%%%%%%%%%%%%%%%%%%%%%%%%%%%%%%%

In this case, when $\cK{}>0$, used should be the arguments outlined in the
respective part of the proof of Lemma~\ref{ewrehgwew}, for $c<\cS$, with the
difference that integrals are analyzed along the lines traced in
Sections~\ref{werfgwegfwe} and \ref{ertherhrhr}. The proof is complete.
\end{proof}

%%%%%%%%%%%%%%%%%%%%%%%%%%%%%%%%%%%%%%%%%%%%%%%%%%%%%%%%%%%%%%%%%%%%%%%%%%%%%%%%%%%%%%%%
\subsection{Step 5: further simplification of the main term of approximation}\label{asfdghjnm}
%%%%%%%%%%%%%%%%%%%%%%%%%%%%%%%%%%%%%%%%%%%%%%%%%%%%%%%%%%%%%%%%%%%%%%%%%%%%%%%%%%%%%%%%

In Step 3 of the proof, the main term of approximation
$\MainApprox{t}(u,c\mid\T{1}=v)$ (see \eqref{dtfyjutykm}) was simplified up to
$\SeqApprox{t}{1}(u,c\mid\T{1}=v)$ (see \eqref{ewdtehfdj}). Let us further
simplify $\SeqApprox{t}{1}(u,c\mid\T{1}=v)$ up to the terms of allowed order of
smallness. We use for it core asymptotic analysis developed in Step 4. It is
noteworthy that in the rest of the proof this analysis is applied only to the
expressions of the second kind.

%%%%%%%%%%%%%%%%%%%%%%%%%%%%%%%%%%%%%%%%%%%%%%%%%%%%%%%%%%%%%%%%%%%%%%%%%%%%%%%%%%%%%%%%
\subsection*{First step in processing (\ref{ewdtehfdj})}\label{sdrtgyjhkmg}
%%%%%%%%%%%%%%%%%%%%%%%%%%%%%%%%%%%%%%%%%%%%%%%%%%%%%%%%%%%%%%%%%%%%%%%%%%%%%%%%%%%%%%%%

Rewrite \eqref{ewdtehfdj} as
\begin{multline*}\label{dtfyjutykmXX}
\SeqApprox{t}{1}(u,c\mid\T{1}=v)=\frac{(\E{Y})^{3/2}(u+cv)^{1/2}}{2\pi
c\sqrt{\D{T}\D{Y}}}\int_{0}^{\frac{c(t-v)}{u+cv}} \frac{1}{(1+x)^{3/2}}
\\[2pt]
\times\sum_{n=\EnOne}^{\infty}n^{-1/2}\sqrt{\frac{(u+cv)(1+x)}{n\E{Y}}}
\exp\big\{-\tfrac{1}{2}\big[\Lambda^2_{n}(u+cv,x)+\Delta_n^2(u+cv,x)\big]\big\}dx
\end{multline*}
and introduce
\begin{multline*}
\SeqApprox{t}{2}(u,c\mid\T{1}=v)=\frac{(\E{Y})^{3/2}(u+cv)^{1/2}}{2\pi
c\sqrt{\D{T}\D{Y}}}\int_{0}^{\frac{c(t-v)}{u+cv}} \frac{1}{(1+x)^{3/2}}
\\[2pt]
\times\,\sum_{n=\EnOne}^{\infty}n^{-1/2}\exp\big\{-\tfrac{1}{2}\big[\Lambda^2_{n}(u+cv,x)
+\Delta_n^2(u+cv,x)\big]\big\}dx.
\end{multline*}

Using Lemma~\ref{asdfhjnmkgh} which yields the identity
\begin{multline*}
1-\sqrt{\frac{(u+cv)(1+x)}{n\E{Y}}}\eqSic
\bigg\{\frac{\sqrt{B_4}}{\sqrt{B_1n}}\Lambda_{n}(u+cv,x)
\\[-4pt]
+\frac{B_3}{\E{Y}\sqrt{B_1n}}\Delta_{n}(u+cv,x)\bigg\}
\bigg(1+\sqrt{\frac{(u+cv)(1+x)}{n\E{Y}}}\,\bigg)^{-1},
\end{multline*}
we have to prove that
\begin{equation*}
\sup_{t>0}\,\Big|\,\SeqApprox{t}{1}(u,c\mid\T{1}=v)-\SeqApprox{t}{2}(u,c\mid\T{1}=v)\,\Big|=
\underline{O}\bigg(\dfrac{\ln(u+cv)}{u+cv}\bigg),
\end{equation*}
as $u+cv\to\infty$. It is done by means of core asymptotic analysis of the
expressions of the second kind described in Step~4. In particular, for this
purpose we have to prove that
\begin{multline*}
(u+cv)^{1/2}\int_{0}^{\frac{c(t-v)}{u+cv}} \frac{1}{(1+x)^{3/2}}
\sum_{n=\EnOne}^{\infty}n^{-1}\big(|\Lambda_{n}(u+cv,x)|
+|\Delta_{n}(u+cv,x)|\big)
\\[2pt]
\times\exp\big\{-\tfrac{1}{2}\big[\Lambda_n^2(u+cv,x)+\Delta^2_{n}(u+cv,x)\big]\big\}dx
=\underline{O}\bigg(\dfrac{\ln(u+cv)}{u+cv}\bigg),
\end{multline*}
as $u+cv\to\infty$. This standard check is left to the reader.

%%%%%%%%%%%%%%%%%%%%%%%%%%%%%%%%%%%%%%%%%%%%%%%%%%%%%%%%%%%%%%%%%%%%%%%%%%%%%%%%%%%%%%%%
\subsection*{Second step in processing (\ref{ewdtehfdj})}\label{edthrjkty}
%%%%%%%%%%%%%%%%%%%%%%%%%%%%%%%%%%%%%%%%%%%%%%%%%%%%%%%%%%%%%%%%%%%%%%%%%%%%%%%%%%%%%%%%

We write
\begin{multline*}
\SeqApprox{t}{3}(u,c\mid\T{1}=v)=\frac{(\E{Y})^{3/2}(u+cv)^{1/2}}{2\pi
c\sqrt{\D{T}\D{Y}}}\int_{0}^{\frac{c(t-v)}{u+cv}} \frac{1}{(1+x)^{3/2}}
\\
\times\exp\Big\{-\tfrac{1}{2}\Delta^2_{{\frac{(u+cv)(1+x)}{\E{Y}}}}(u+cv,x)\Big\}
\,\sum_{n=\EnOne}^{\infty}n^{-1/2}
\exp\big\{-\tfrac{1}{2}\Lambda_n^2(u+cv,x)\big\}dx.
\end{multline*}
We have to prove that
\begin{equation*}
\sup_{t>0}\,\Big|\,\SeqApprox{t}{2}(u,c\mid\T{1}=v)-\SeqApprox{t}{3}(u,c\mid\T{1}=v)\,\Big|=
\underline{O}\bigg(\dfrac{\ln(u+cv)}{u+cv}\bigg),
\end{equation*}
as $u+cv\to\infty$. It is done by means of core asymptotic analysis of the
expressions of the second kind described in Step~4. This standard check is left
to the reader.

%%%%%%%%%%%%%%%%%%%%%%%%%%%%%%%%%%%%%%%%%%%%%%%%%%%%%%%%%%%%%%%%%%%%%%%%%%%%%%%%%%%%%%%%
\subsection*{Third step in processing (\ref{ewdtehfdj})}\label{weruyhjtyk}
%%%%%%%%%%%%%%%%%%%%%%%%%%%%%%%%%%%%%%%%%%%%%%%%%%%%%%%%%%%%%%%%%%%%%%%%%%%%%%%%%%%%%%%%

Bearing in mind the identity\footnote{Note that
$1-\sqrt{1+x}=-x/2+x^2/8-x^3/16+\dots$} (see Lemma~\ref{asdffgfh})
\begin{equation*}
\Lambda_{n+1}(u+cv,x)-\Lambda_{n}(u+cv,x)=\left(\frac{B_1}{B_4n}\right)^{1/2}+\Lambda_{n+1}(u+cv,x)
\big(1-\sqrt{1+1/n}\big),
\end{equation*}
we prove by means of core asymptotic analysis of the expressions of the second
kind described in Step~4 that
\begin{equation*}
\sup_{t>0}\,\Big|\,\SeqApprox{t}{3}(u,c\mid\T{1}=v)-\SeqApprox{t}{4}(u,c\mid\T{1}=v)\,\Big|
=\underline{O}\bigg(\dfrac{\ln(u+cv)}{u+cv}\bigg),
\end{equation*}
as $u+cv\to\infty$, where
\begin{multline*}
\SeqApprox{t}{4}(u,c\mid\T{1}=v)=\frac{(u+cv)^{1/2}(\E{Y})^{3/2}}{\sqrt{2\pi}
c\sqrt{B_1}} \int_{0}^{\frac{c(t-v)}{(u+cv)}}\frac{1}{(1+x)^{3/2}}
\\
\times\exp\big\{-\tfrac{1}{2}
\big[\Delta^2_{\frac{(u+cv)(1+x)}{\E{Y}}}(u+cv,x)\big]\big\}
\frac{1}{\sqrt{2\pi}}\sum_{n=\EnOne}^{\infty}\big(\Lambda_{n+1}(u+cv,x)
\\
-\Lambda_{n}(u+cv,x)\big)
\exp\Big\{-\tfrac{1}{2}\big[\Lambda_n^2(u+cv,x)\big]\Big\}\,dx.
\end{multline*}
This standard check is left to the reader.

%%%%%%%%%%%%%%%%%%%%%%%%%%%%%%%%%%%%%%%%%%%%%%%%%%%%%%%%%%%%%%%%%%%%%%%%%%%%%%%%%%%%%%%%
\subsection*{Fourth step in processing (\ref{ewdtehfdj})}\label{sfdgdhfjkg}
%%%%%%%%%%%%%%%%%%%%%%%%%%%%%%%%%%%%%%%%%%%%%%%%%%%%%%%%%%%%%%%%%%%%%%%%%%%%%%%%%%%%%%%%

We finally note that $D^2=B_1/(\E{Y})^3$ and that
\begin{equation*}
\Delta^2_{\frac{(u+cv)(1+x)}{\E{Y}}}(u+cv,x)
=\dfrac{(x-(1+x)(c/\cS))^2}{c^2D^2\frac{(1+x)}{(u+cv)}}
%\\
=(u+cv)\dfrac{(x[1/c-1/\cS]-1/\cS)^2}{D^2(1+x)}.
\end{equation*}
By means of standard core asymptotic analysis of the expressions of the second
kind described in Step~4, we prove that
\begin{equation*}
\sup_{t>0}\,\Big|\,\SeqApprox{t}{4}(u,c\mid\T{1}=v)-\SeqApprox{t}{5}(u,c\mid\T{1}=v)\,\Big|
=\underline{O}\bigg(\dfrac{\ln(u+cv)}{u+cv}\bigg),
\end{equation*}
as $u+cv\to\infty$, where
\begin{multline*}
\SeqApprox{t}{5}(u,c\mid\T{1}=v)=\frac{(u+cv)^{1/2}}{\sqrt{2\pi c^2D^2}}
\int_{0}^{\frac{c(t-v)}{(u+cv)}}\frac{1}{(1+x)^{3/2}}
\\
\times\exp\bigg\{-\frac{1}{2}
\frac{(x-(1+x)(c/\cS))^2}{c^2D^2\frac{(1+x)}{(u+cv)}}\bigg\}\,dx,
\end{multline*}
which yields the required approximation. The proof of Theorem \ref{srdthjrf} is
complete.

%%%%%%%%%%%%%%%%%%%%%%%%%%%%%%%%%%%%%%%%%%%%%%%%%%%%%%%%%%%%%%%%%%%%%%%%%%%%%%%%%%%%%%%%
\section{Main technicalities and auxiliary results}\label{wqedtfyjhtgfj}
%%%%%%%%%%%%%%%%%%%%%%%%%%%%%%%%%%%%%%%%%%%%%%%%%%%%%%%%%%%%%%%%%%%%%%%%%%%%%%%%%%%%%%%%

In this section, we gather main auxiliary results used in Section
\ref{sdfyguklyi}.

%%%%%%%%%%%%%%%%%%%%%%%%%%%%%%%%%%%%%%%%%%%%%%%%%%%%%%%%%%%%%%%%%%%%%%%%%%%%%%%%%%%%%%%%
\subsection{Non-uniform Berry-Esseen bounds in local CLT}\label{srdthrj}
%%%%%%%%%%%%%%%%%%%%%%%%%%%%%%%%%%%%%%%%%%%%%%%%%%%%%%%%%%%%%%%%%%%%%%%%%%%%%%%%%%%%%%%%

Let the random vectors $\xi_{i}$, $i=1,2,\dots$, assuming values in $\R^{m}$ be
i.i.d. with c.d.f. $P$, with zero mean and with identity covariance matrix $I$.
Put $S_n=\frac{1}{\sqrt{n}}\sum_{i=1}^{n}\xi_{i}$, $\P_{n}(A)=\P\{S_{n}\in
A\}$, $A\subset\R^{m}$, $\p_n(x)=\frac{\partial^m}{\partial x_1,\dots,\partial
x_m}\P\{S_n\leqslant x\}$, $x=(x_1,\dots,x_m)\in\R^{m}$.

The Berry-Esseen bounds in one-dimensional, as $m=1$, central limit theorem
(CLT) are well known. The following theorem follows from Theorem~11 in Ch. 7,
\S~2 of \citeNP{[Petrov 1975]} proved for non-identically distributed random
variables $\xi_{i}$, $i=1,2,\dots$.
\begin{theorem}[\citeNP{[Petrov 1975]}]\label{qwretyuhjrk}
Let $\E\xi_1^2>0$, $\E|\xi_1|^3<\infty$, and
$\int_{|t|>\epsilon}|\E{e^{it\xi_1}}|^ndt=\underline{O}(n^{-1})$ for any fixed
$\epsilon>0$. Then for all sufficiently large $n$ a bounded p.d.f. $\p_n(x)$
exists and
\begin{equation*}
\sup_{x\in\R}\left|\,\p_n(x)-\Ugauss{0}{1}(x)\right|=\underline{O}(n^{-1/2}),\quad
n\to\infty.
\end{equation*}
\end{theorem}

The non-uniform Berry-Esseen bounds in integral rather than local
one-dimensional CLT may be found in \citeNP{[Petrov 1995]} (see, e.g., Theorems
15 and 14 in Ch. 5, \S~6 in \citeNP{[Petrov 1995]}).

A detailed study of normal approximations and asymptotic expansions in the CLT
in $\R^{m}$, as $m>1$, is conducted in \citeNP{[Bhattacharya Rao 1976]} (see
particularly Theorem 19.2 in \citeNP{[Bhattacharya Rao 1976]}. The non-uniform
Berry-Esseen bounds in $\R^{m}$, $m>1$, that is used in
Section~\ref{sdfyguklyi} as auxiliary result, is Theorem~4 in \S~3 of
\citeNP{[Dubinskaite 1982]} with $k=m$ and $s=2$. We first formulate the
following conditions.

\emph{Condition} ($P_m$): there exists $N\geqslant 1$ such that
$\sup_{x\in\R^m}p_N(x)\leqslant C<\infty$ and
\begin{equation*}
%\psi_{2,n}=
\int_{\|x\|>\sqrt{n}}\|x\|^{2}P(dx)+\frac{1}{n}\int_{\|x\|\leqslant\sqrt{n}}\|x\|^{4}P(dx)
+\frac{1}{\sqrt{n}}\sup_{\|e\|=1}\Big|\int_{\|x\|\leqslant\sqrt{n}}(x,e)^{3}P(dx)\Big|
=\underline{O}(\epsilon_n),
\end{equation*}
$n\to\infty$, where $\epsilon_n$ is a sequence of positive numbers such that
$\epsilon_n\to 0$, as $n\to\infty$, and $\epsilon_n\geqslant 1/\sqrt{n}$.

\emph{Condition} ($A_2$): $\beta_{2}=\E\|\xi_1\|^2<\infty$,
$\alpha_1(t)=\E(\xi_1,t)<\infty$.

\begin{theorem}[\citeNP{[Dubinskaite 1982]}]\label{w4e5y46yu43}
To have
\begin{equation}\label{ewyjhfyhm}
(1+\|x\|)^3\left|\,\p_n(x)-\Ugauss{0}{I}(x)\right|=\underline{O}(n^{-1/2}),\quad
n\to\infty,
\end{equation}
it is necessary and sufficient that conditions $(P_m)$, $(A_2)$, and
\begin{equation*}
z\int_{\|x\|>z}\|x\|^{2}P(dx)+\sup_{\|e\|=1}\Big|\int_{\|x\|\leqslant
z}(x,e)^{3}P(dx)\Big|=\underline{O}(1),\quad z\to\infty,
\end{equation*}
be satisfied.
\end{theorem}

\begin{remark}
It is known that the estimate \eqref{ewyjhfyhm} is optimal in terms of
dependence on $\|x\|$, i.e., the power $3$ in \eqref{ewyjhfyhm} can not be
replaced by a greater one.
\end{remark}

%%%%%%%%%%%%%%%%%%%%%%%%%%%%%%%%%%%%%%%%%%%%%%%%%%%%%%%%%%%%%%%%%%%%%%%%%%%%%%%%%%%%%%%%
\subsection{Large deviations for sums of i.i.d. r.v.}\label{wqeujmyr}
%%%%%%%%%%%%%%%%%%%%%%%%%%%%%%%%%%%%%%%%%%%%%%%%%%%%%%%%%%%%%%%%%%%%%%%%%%%%%%%%%%%%%%%%

The following theorem is Corollary 2 in \citeNP{[Nagaev 1965]} (see also
\citeNP{[Nagaev 1979]}).

\begin{lemma}\label{q1werhewef}
Let $\xi_i$, $i=1,2,\dots$, be i.i.d. random variables such that $\E\xi_1=0$
and $\D\xi=1$. If $c_m=\E|\xi_i|^m<\infty$ with $m>2$, then for
$x>4\sqrt{n\max\big[\ln\big(\frac{n^{m/2-1}}{K_m c_m}\big),0\big]}$
\begin{equation*}
\P\Bigg\{\sum_{i=1}^n\xi_i>x\Bigg\}<\frac{B_m c_m n}{x^m},
\end{equation*}
where $K_m=1+(m+1)^{m+2}e^{-m}$, and $B_m$ is an absolute constant depending
only on $m$.
\end{lemma}

%%%%%%%%%%%%%%%%%%%%%%%%%%%%%%%%%%%%%%%%%%%%%%%%%%%%%%%%%%%%%%%%%%%%%%%%%%%%%%%%%%%%%%%%
\subsection{Fundamental identities}\label{werwhteyjn}
%%%%%%%%%%%%%%%%%%%%%%%%%%%%%%%%%%%%%%%%%%%%%%%%%%%%%%%%%%%%%%%%%%%%%%%%%%%%%%%%%%%%%%%%

For $B_1=(\E{T})^2\D{Y}+(\E{Y})^2\D{T}$, $B_2=\E{Y}\D{T}$, $B_3=\E{T}\D{Y}$,
and $B_4=\D{Y}\D{T}$, we use notation
\begin{equation}\label{wqerertjr4}
\begin{aligned}
\mathcal{Y}_{n}(\OneVar)&=\frac{\OneVar-n\E{Y}}{\sqrt{n\D{Y}}},&\mathcal{T}_{n}(\TwoVar)&=\frac{\TwoVar-n\E{T}}{\sqrt{n\D{T}}},
\\
\Delta_{n}(\OneVar,\TwoVar)&=\dfrac{\TwoVar\E{Y}-\OneVar\E{T}}{\sqrt{B_1n}},&
\Lambda_{n}(\OneVar,\TwoVar)&=\dfrac{B_1n-(B_2\OneVar+B_3\TwoVar)}{\sqrt{B_1B_4n}}.
\end{aligned}
\end{equation}

\begin{lemma}\label{erwtjytk}
We have the identity
\begin{equation*}
\mathcal{Y}^2_{n}(\OneVar)+\mathcal{T}^2_{n}(\TwoVar)=\Delta^2_{n}(\OneVar,\TwoVar)+\Lambda_n^2(\OneVar,\TwoVar).
\end{equation*}
\end{lemma}

\begin{proof}
Getting of this identity is based on algebraic manipulations with the left-hand
side, aimed at completing the square. Its proof may be done as well by means of
a straightforward check.
\end{proof}

\begin{lemma}\label{asdffgfh}
We have the identity
\begin{equation*}
\Lambda_{n+1}(\OneVar,\TwoVar)-\Lambda_{n}(\OneVar,\TwoVar)=\left(\frac{B_1}{B_4n}\right)^{1/2}
+\Lambda_{n+1}(\OneVar,\TwoVar)\big(1-\sqrt{1+1/n}\big).
\end{equation*}
\end{lemma}

\begin{proof}
We have
\begin{multline*} (B_1B_4n)^{1/2}[\Lambda_{n+1}(\OneVar,\TwoVar)-\Lambda_{n}(\OneVar,\TwoVar)]
\\
=\big\{(B_1B_4(n+1))^{1/2}\Lambda_{n+1}(\OneVar,\TwoVar)-(B_1B_4n)^{1/2}\Lambda_{n}(\OneVar,\TwoVar)\big\}
\\
+\big\{(B_1B_4n)^{1/2}\Lambda_{n+1}(\OneVar,\TwoVar)-(B_1B_4(n+1))^{1/2}\Lambda_{n+1}(\OneVar,\TwoVar)\big\}
\\
=B_1+\Lambda_{n+1}(\OneVar,\TwoVar)(B_1B_4n)^{1/2}\big(1-\sqrt{1+1/n}\big).
\end{multline*}
Indeed, since
\begin{equation*}
(B_1B_4(n+1))^{1/2}\Lambda_{n+1}(\OneVar,\TwoVar)=B_1(n+1)-(B_2\OneVar+B_3\TwoVar),
\end{equation*}
\begin{equation*}
(B_1B_4n)^{1/2}\Lambda_{n}(\OneVar,\TwoVar)=B_1n-(B_2\OneVar+B_3\TwoVar),
\end{equation*}
the first summand is
\begin{multline*}
(B_1B_4(n+1))^{1/2}\Lambda_{n+1}(\OneVar,\TwoVar)-(B_1B_4n)^{1/2}\Lambda_{n}(\OneVar,\TwoVar)
\\
=B_1(n+1)-(B_2\OneVar+B_3\TwoVar)-B_1n+(B_2\OneVar+B_3\TwoVar)=B_1.
\end{multline*}
The second summand is
\begin{multline*}
(B_1B_4n)^{1/2}\Lambda_{n+1}(\OneVar,\TwoVar)-(B_1B_4(n+1))^{1/2}\Lambda_{n+1}(\OneVar,\TwoVar)
\\[4pt]
=(B_1B_4n)^{1/2}\Lambda_{n+1}(\OneVar,\TwoVar)\bigg\{1-\frac{(B_1B_4(n+1))^{1/2}}{(B_1B_4n)^{1/2}}\bigg\}
\\[2pt]
=(B_1B_4n)^{1/2}\Lambda_{n+1}(\OneVar,\TwoVar)\bigg\{1-\sqrt{\frac{(n+1)}{n}}\,\bigg\}.
\end{multline*}
The proof is complete.
\end{proof}

\begin{lemma}\label{asdfhjnmkgh}
We have the identities
\begin{equation*}
1-\frac{\OneVar}{n\E{Y}}=\frac{\sqrt{B_4}}{\sqrt{B_1n}}\Lambda_{n}(\OneVar,\TwoVar)
+\frac{B_3}{\E{Y}\sqrt{B_1n}}\Delta_{n}(\OneVar,\TwoVar)
\end{equation*}
and
\begin{equation*}
1-\sqrt{\frac{\OneVar}{n\E{Y}}}=
%\eqErr
%(??)=
\bigg\{\frac{\sqrt{B_4}}{\sqrt{B_1n}}\Lambda_{n}(\OneVar,\TwoVar)
+\frac{B_3}{\E{Y}\sqrt{B_1n}}\Delta_{n}(\OneVar,\TwoVar)\bigg\}\bigg(1+\sqrt{\frac{\OneVar}{n\E{Y}}}\bigg)^{-1}.
\end{equation*}
\end{lemma}

\begin{proof}
Bearing in mind that $B_1/\E{Y}-B_2=(\E{T})^2\D{Y}/\E{Y}$, we have
\begin{multline*}
\Lambda_{n}(\OneVar,\TwoVar)=\left(1-\frac{\OneVar}{\E{Y}n}\right)\left(\dfrac{B_1n}{B_4}\right)^{1/2}+\frac{\OneVar
B_1/\E{Y}}{\sqrt{B_1B_4n}} -\dfrac{B_2\OneVar+B_3\TwoVar}{\sqrt{B_1B_4n}}
\\
=\left(1-\frac{\OneVar}{\E{Y}n}\right)\left(\dfrac{B_1n}{B_4}\right)^{1/2}
-\frac{\E{T}\D{Y}}{\E{Y}\sqrt{B_4}}\left(\frac{\TwoVar\E{Y}-\OneVar\E{T}}{\sqrt{B_1n}}\right)
\\
=\left(1-\frac{\OneVar}{\E{Y}n}\right)\left(\dfrac{B_1n}{B_4}\right)^{1/2}
-\frac{\E{T}\D{Y}}{\E{Y}\sqrt{B_4}}\,\Delta_{n}(\OneVar,\TwoVar).
\end{multline*}
Rewrite it
\begin{multline*}
1-\frac{\OneVar}{\E{Y}n}=\left(\dfrac{B_4}{B_1n}\right)^{1/2}
\left(\Lambda_{n}(\OneVar,\TwoVar)+\frac{\E{T}\D{Y}}{\E{Y}\sqrt{B_4}}\Delta_{n}(\OneVar,\TwoVar)\right)
\\
=\left(\dfrac{1}{B_1n}\right)^{1/2}\left(B_4^{1/2}\Lambda_{n}(\OneVar,\TwoVar)+\frac{B_3}{\E{Y}}\Delta_{n}(\OneVar,\TwoVar)\right),
\end{multline*}
as required. The proof is complete.
\end{proof}

\begin{remark}
The identities of Lemmas \ref{erwtjytk}--\ref{asdfhjnmkgh} in a more general
form were proved and used first in \citeNP{[Malinovskii 1993]}.
\end{remark}

%%%%%%%%%%%%%%%%%%%%%%%%%%%%%%%%%%%%%%%%%%%%%%%%%%%%%%%%%%%%%%%%%%%%%%%%%%%%%%%%%%%%%%%%
\subsection{Sums related to zeta-functions and polygamma functions}\label{qwerjktyk}
%%%%%%%%%%%%%%%%%%%%%%%%%%%%%%%%%%%%%%%%%%%%%%%%%%%%%%%%%%%%%%%%%%%%%%%%%%%%%%%%%%%%%%%%

For real $s>1$ and for integer $N>0$, fairly easy is the upper bound
$\sum_{n>N}\frac{1}{n^s}\leqslant\frac{1}{N^s}+\int_{N}^{\infty}\frac{du}{u^{s}}
=\frac{1}{N^s}+\frac{1}{s-1}N^{1-s}\leqslant\frac{s}{s-1}N^{1-s}$. Much more
accurate are the following equalities well known in the theory of Riemann
zeta-function and its generalizations.

%%%%%%%%%%%%%%%%%%%%%%%%%%%%%%%%%%%%%%%%%%%%%%%%%%%%%%%%%%%%%%%%%%%%%%%%%%%%%%%%%%%%%%%%
\subsection*{Sums related to Riemann zeta-function}\label{wqerthjtr}
%%%%%%%%%%%%%%%%%%%%%%%%%%%%%%%%%%%%%%%%%%%%%%%%%%%%%%%%%%%%%%%%%%%%%%%%%%%%%%%%%%%%%%%%

By Riemann zeta-function with $s>1$, we call
\begin{equation*}
\zeta(s)=\sum_{n=1}^{\infty}\frac{1}{n^s}.
\end{equation*}
It is known that
\begin{equation*}
\sum_{n=N+1}^{\infty}\frac{1}{n^s}=\frac{N^{1-s}}{s-1}
-\frac{1}{2}N^{-s}+s\int_{N}^{\infty}\frac{\rho(u)du}{u^{s+1}},
\end{equation*}
where $\rho(x)=\frac12-\{x\}$, and for $M>N$
\begin{multline*}
\sum_{N+\frac32<n\leqslant
M+\frac32}\frac{1}{n^s}=\int_{N-\frac12}^{M+\frac12}\frac{du}{(u+1)^{s}}
+s\int_{N-\frac12}^{M+\frac12}\frac{\rho(u)du}{(u+1)^{s+1}}
\\
=\frac{(M+\tfrac32)^{1-s}}{1-s}-\frac{(N+\tfrac12)^{1-s}}{1-s}
+s\int_{N+\tfrac12}^{M+\frac32}\frac{\rho(u)du}{u^{s+1}}.
\end{multline*}
The former equality is explicit as Corollary 2 in Ch.~1, \S~4 of
\citeNP{[Karatsuba Voronin 1992]}, the latter is shown in the proof of Lemma 3
in Ch.~1, \S~4 of \citeNP{[Karatsuba Voronin 1992]}.

%%%%%%%%%%%%%%%%%%%%%%%%%%%%%%%%%%%%%%%%%%%%%%%%%%%%%%%%%%%%%%%%%%%%%%%%%%%%%%%%%%%%%%%%
\subsection*{Sums related to Hurwitz zeta-function}\label{tyjtjerj}
%%%%%%%%%%%%%%%%%%%%%%%%%%%%%%%%%%%%%%%%%%%%%%%%%%%%%%%%%%%%%%%%%%%%%%%%%%%%%%%%%%%%%%%%

By Hurwitz zeta-function with $s>1$ and $x>0$, we call
\begin{equation*}
\zeta(s,x)=\sum_{n=x}^{\infty}\frac{1}{n^s}=\sum_{n=0}^{\infty}\frac{1}{(n+x)^s}.
\end{equation*}
For $x>0$ and for any
$s\ne 1$, a convergent Newton series representation is known:
\begin{equation*}
\zeta(s,x)=\sum_{n=0}^{\infty}\frac{1}{(n+x)^s}=\frac{1}{s-1}\sum_{n=0}^{\infty}\frac{1}{n+1}
\sum_{k=0}^{n}(-1)^k\binom{n}{k} (x+k)^{1-s}\sim\frac{x^{1-s}}{s-1}.
\end{equation*}
It is easily seen that $\frac{\partial}{\partial
s}\zeta(s,x)=-\sum_{n=0}^{\infty}\frac{\ln(n+x)}{(n+x)^s}=
\sum_{n=x}^{\infty}\frac{\ln(n)}{n^s}$, and we have
\begin{multline}\label{sdfgrweghsfbh}
\sum_{n=x}^{\infty}\frac{\ln(n)}{n^s}\sim
-\Big(\ln(x)+\frac{1}{s-1}\Big)\zeta(s,x)+\frac{1}{2(s-1)x^s}+\dots
\\
=-\Big(\ln(x)+\frac{1}{s-1}\Big)\frac{x^{1-s}}{s-1}+\dots,\ x\to\infty.
\end{multline}

%%%%%%%%%%%%%%%%%%%%%%%%%%%%%%%%%%%%%%%%%%%%%%%%%%%%%%%%%%%%%%%%%%%%%%%%%%%%%%%%%%%%%%%%
\subsection*{Sums related to polygamma functions}\label{wertgjkty}
%%%%%%%%%%%%%%%%%%%%%%%%%%%%%%%%%%%%%%%%%%%%%%%%%%%%%%%%%%%%%%%%%%%%%%%%%%%%%%%%%%%%%%%%

By polygamma function, we call
$\psi(x)=\Gamma^{\prime}(x)/\Gamma(x)\sim\ln(x)$. By polygamma function of
order $m\geqslant 1$, we call $\psi^{(m)}(x)=\frac{d^m}{dx^m}\psi(x)$. It is
known that for $x\to+\infty$
\begin{equation*}
\psi^{(m)}(x)\sim(-1)^{m+1}\sum_{k=0}^{\infty}\frac{(k+m-1)!}{k!}\frac{B_k}{x^{k+m}},
\end{equation*}
where $B_k$ are Stirling's numbers with $B_1=1/2$. It is known that
$\sum_{n=N}^{M}\frac{1}{n}=\psi(1+M)-\psi(1+N)$ for $N<M$. Consequently, we
have, e.g.,
\begin{equation*}
\sum_{n=U}^{U^2}\frac{1}{n}\sim\ln(U)+\frac{1}{2U}+\frac{7}{12U^2}+\dots,\
U\to\infty.
\end{equation*}

Using polygamma functions, we can get the explicit expressions and exact asymptotics of a
number of series of this type. In particular, for $N<M$ we have
$\sum_{n=N}^{M}\frac{1}{n(n+N)}=\frac{1}{N}\big(\psi(1+M)
-\psi(N)+\psi(2N)-\psi(1+M+N)\big)$ and
\begin{equation*}
\sum_{n=U}^{U^2}\frac{1}{n(n+U)}=\sum_{n=\epsilon
U}^{KU^2}\frac{1}{n^2}\bigg(\frac{1}{1+\frac{U}{n}}\bigg)\sim\frac{\ln(2)}{U}-\frac{3}{4U^2}+\frac{9}{16
U^3}+\dots,\ U\to\infty.
\end{equation*}

%%%%%%%%%%%%%%%%%%%%%%%%%%%%%%%%%%%%%%%%%%%%%%%%%%%%%%%%%%%%%%%%%%%%%%%%%%%%%%%%%%%%%%%%
\subsection{Integrals of rational functions}\label{wqretgherh}
%%%%%%%%%%%%%%%%%%%%%%%%%%%%%%%%%%%%%%%%%%%%%%%%%%%%%%%%%%%%%%%%%%%%%%%%%%%%%%%%%%%%%%%%

The following integrals of rational functions modified by a square root (cf.
3.158 in \citeNP{[Gradshtein 1980]}) can be found in explicit form. We leave to
the reader the details of these calculations.

\begin{lemma}\label{rthfgnmhfgj}
For $L+R>0$, $M>0$, we have
\begin{multline*}
\int_{L}^{\infty}(y+R)^{-1}(y^2+M^2)^{-3/2}dy\eqOK\frac{R}{M^2(R^2+M^2)}
-\frac{M^2+LR}{M^2\sqrt{L^2+M^2}(R^2+M^2)}
\\[4pt]
+\frac{1}{(R^2+M^2)^{3/2}}\ln\left(\frac{M^2-LR+\sqrt{L^2+M^2}\sqrt{R^2+M^2}}{(L+R)(\sqrt{R^2+M^2}-R)}\right).
\end{multline*}
\end{lemma}

\begin{lemma}\label{wqetfyuk6}
For $L\geqslant 0$, $L+R>0$, $M>0$, we have
\begin{multline*}
\int_{L}^{\infty}\mid
y\mid(y+R)^{-1}(y^2+M^2)^{-3/2}dy\eqOK\frac{1}{R^2+M^2}+\frac{R-L}{(R^2+M^2)\sqrt{M^2+L^2}}
\\
+\frac{R}{(R^2+M^2)^{3/2}}\ln\left(\frac{(R+L)(\sqrt{R^2+M^2}-R)}{M^2-RL+\sqrt{R^2+M^2}\sqrt{M^2+L^2}}\right),
\end{multline*}
and for $L\leqslant 0$, $L+R>0$, $M>0$, we have
\begin{multline*}
\int_{L}^{\infty}|y|(y+R)^{-1}(y^2+M^2)^{-3/2}dy\eqOK\frac{2R+M}{M(R^2+M^2)}-\frac{R-L}{(R^2+M^2)\sqrt{M^2+L^2}}
\\
\hskip -6pt+\frac{R}{(R^2+M^2)^{3/2}}\ln\left(\frac{R^2(\sqrt{R^2+M^2}-R)(M^2-RL+\sqrt{R^2+M^2}\sqrt{M^2+L^2})}
{M^2(M+\sqrt{R^2+M^2})^2(R+L)}\right).
\end{multline*}
\end{lemma}

To shorten notation in the following two lemmas, we put $P=-\frac{L}{2M}$,
$K=2MR$.

\begin{lemma}\label{wqerthrjnmr}
For $K>0$, we have
\begin{equation*}
\int_{0}^{\infty}\big(K+y\big)^{-1}
\big(1+\sqrt{y}\big)^{-3}dy\eqOK\frac{K-3}{(1+K)^2}-\frac{\pi\sqrt{K}(K-3)}{(1+K)^3}+\frac{3K-1}{(1+K)^3}\ln(K),
\end{equation*}
and for $0<P<K$, we have
\begin{multline*}
\int_{0}^{P}\big(K-y\big)^{-1}
\big(1+\sqrt{y}\big)^{-3}dy\eqOK\frac{4\sqrt{P}+(3+K)P}{(K-1)^2(1+\sqrt{P})^2}
+\frac{\sqrt{K}(3+K)}{(K-1)^3}\ln\bigg(\frac{\sqrt{K}+\sqrt{P}}{\sqrt{K}-\sqrt{P}}\bigg)
\\
+\frac{3K+1}{(K-1)^3}\ln\bigg(\frac{K-P}{K(1+\sqrt{P})^2}\bigg).
\end{multline*}
\end{lemma}

\begin{lemma}\label{eruyjtky}
For $K>0$, $P<0$, we have
\begin{multline*}
\int_{P}^{\infty}|y|\big(K+y\big)^{-1}\big(1+\sqrt{|y|}\big)^{-3}dy
\eqOK\frac{5K+1}{(1+K)^2}+\frac{\pi K^{3/2}(K-3)}{(1+K)^3}+\frac{5K-1}{(K-1)^2}
\\
+\frac{2\sqrt{P}(1-3K)-(5K-1)}{(K-1)^2(1+\sqrt{P})^2}
-\frac{K(3K-1)}{(1+K)^3}\ln(K)
\\
+\frac{K^{3/2}(3+K)}{(K-1)^3}
\ln\left(\frac{\sqrt{P}+\sqrt{K}}{\sqrt{K}-\sqrt{P}}\right)
+\frac{K(3K+1)}{(K-1)^3}\ln\left(\frac{K-P}{K(1+\sqrt{P})^2}\right).
\end{multline*}
\end{lemma}

%%%%%%%%%%%%%%%%%%%%%%%%%%%%%%%%%%%%%%%%%%%%%%%%%%%%%%%%%%%%%%%%%%%%%%%%%%%%%%%%%%%%%%%%
\subsection{The asymptotic behavior of $\Integral{1}$}\label{qasrdtjkyl}
%%%%%%%%%%%%%%%%%%%%%%%%%%%%%%%%%%%%%%%%%%%%%%%%%%%%%%%%%%%%%%%%%%%%%%%%%%%%%%%%%%%%%%%%

Let us verify that for $R=\mcA\tfrac{U}{\sqrt{n}}+\mcB\sqrt{n}$,
$M=\mcC\tfrac{U}{\sqrt{n}}$,
$L=\tfrac{\mcC^2}{\mcA}\tfrac{U}{\sqrt{n}}-\mcB\sqrt{n}$ with
$\mcA,\mcB,\mcC>0$
\begin{equation*}
\Integral{1}=U\hskip -8pt\sum_{\EnOne<n<\frac{(\E{T})^2}{B_1}\,U^2}\hskip -6pt
n^{-3/2}\int_{L}^{\infty}\big(y+R\big)^{-1}
\big(y^2+M^2\big)^{-3/2}\,dy=\underline{O}\bigg(\dfrac{\ln(U)}{U}\bigg),\
U\to\infty.
\end{equation*}

Put the above $R$, $M$, and $L$ in the integral evaluated in
Lemma~\ref{rthfgnmhfgj}. It is checked by direct calculations that
\begin{equation*}
\begin{aligned}
&\frac{Un^{-3/2}R}{M^2(R^2+M^2)}\eqOK\frac{1}{\mcC^2U}\frac{\mcB
n+\mcA U}{(\mcB{n}+\mcA U)^2+\mcC^2 U^2},
\\[10pt]
&\frac{Un^{-3/2}(M^2+LR)}{M^2\sqrt{L^2+M^2}\,(R^2+M^2)}\eqOK\frac{1}{\mcC^2U}\frac{\mcA\mcC^2 U^2
-(\mcB n+\mcA U)(\mcA \mcB n-\mcC^2U)}{((\mcB n+\mcA U)^2+\mcC^2U^2)
\sqrt{(\mcA\mcB n-\mcC^2U)^2+\mcA^2\mcC^2U^2}},
\\[10pt]
&\frac{Un^{-3/2}}{(R^2+M^2)^{3/2}}
\ln\bigg(\frac{M^2-LR+\sqrt{L^2+M^2}\sqrt{R^2+M^2}}{(L+R)\,\big(\sqrt{R^2+M^2}-R\,\big)}\bigg)
=\frac{U\ln(\NomiN{1}/\DeNomiN{1})}{((\mcB n+\mcA U)^2+\mcC^2U^2)^{3/2}},
\end{aligned}
\end{equation*}
where
\begin{equation*}
\begin{aligned}
\NomiN{1}&=\mcB(\mcC^2-\mcA^2)nU-\mcA\mcB^2n^2-((\mcB n+\mcA
U)^2+\mcC^2U^2)^{1/2}((\mcA\mcB n-\mcC^2U)^2+\mcA^2\mcC^2U^2)^{1/2},
\\[2pt]
\DeNomiN{1}&=(\mcA^2+\mcC^2)U(\mcB n+\mcA U-((\mcB n+\mcA
U)^2+\mcC^2U^2)^{1/2}).
\end{aligned}
\end{equation*}

We have $\Integral{1}=\IntegraL{1}{1}+\IntegraL{1}{2}+\IntegraL{1}{3}$, where
\begin{equation*}
\begin{aligned}
\IntegraL{1}{1}&=\frac{1}{\mcC^2
U}\sum_{\EnOne<n<\frac{(\E{T})^2}{B_1}\,U^2}\frac{\mcB n+\mcA U}{(\mcB{n}+\mcA
U)^2+\mcC^2U^2},
\\
\IntegraL{1}{2}&=\frac{1}{\mcC^2U}\sum_{\EnOne<n<\frac{(\E{T})^2}{B_1}\,U^2}\frac{\mcA\mcC^2
U^2 -(\mcB n+\mcA U)(\mcA \mcB n-\mcC^2U)}{((\mcB n+\mcA U)^2+\mcC^2U^2)
\sqrt{(\mcA\mcB n-\mcC^2U)^2+\mcA^2\mcC^2U^2}},
\\
\IntegraL{1}{3}&=U\sum_{\EnOne<n<\frac{(\E{T})^2}{B_1}\,U^2}
\frac{\ln({\NomiN{1}}/{\DeNomiN{1}})}{((\mcB n+\mcA U)^2+\mcC^2U^2)^{3/2}}.
\end{aligned}
\end{equation*}
To investigate the asymptotic behavior of $\IntegraL{1}{1}$, $\IntegraL{1}{2}$,
$\IntegraL{1}{3}$, as $U\to\infty$, note that the fractions under the summation
sign, as well as the argument of the logarithmic function in $\IntegraL{1}{3}$,
are rational functions of $n$ modified by a square root. Extracting the highest
power of $n$ from both nominators and denominators of these fractions, we have
\begin{equation*}
\begin{aligned}
\IntegraL{1}{1}&=\frac{1}{\mcC^2
U}\sum_{\EnOne<n<\frac{(\E{T})^2}{B_1}\,U^2}\frac{1}{n}\,
\underbrace{\frac{\mcB+\mcA\frac{U}{n}}{\mcB^2+2\mcA\mcB\frac{U}{n}+(\mcA^2+\mcC^2)
\frac{U^2}{n^2}}},
\\[2pt]
\IntegraL{1}{2}
&=\frac{1}{\mcC^2U}\sum_{\EnOne<n<\frac{(\E{T})^2}{B_1}\,U^2}\frac{1}{n}
\,\underbrace{\frac{\mcA\mcC^2\frac{U^2}{n^2}
-(\mcB+\mcA\frac{U}{n})(\mcA\mcB-\mcC^2\frac{U}{n})}
{\big((\mcB+\mcA\frac{U}{n})^2+\mcC^2\frac{U^2}{n^2}\big)
\sqrt{\big(\mcA\mcB-\mcC^2\frac{U}{n}\big)^2+\mcA^2\mcC^2\frac{U^2}{n^2}}}},
\\[2pt]
\IntegraL{1}{3}&=U\sum_{\EnOne<n<\frac{(\E{T})^2}{B_1}\,U^2}\frac{\ln\big(n\,
(\NomiN{1}/n^2)/(\DeNomiN{1}/n)
\big)}{n^3}\underbrace{\frac{1}{\big((\mcB+\mcA\frac{U}{n})^2
+\mcC^2\frac{U^2}{n^2}\big)^{3/2}}}.
\end{aligned}
\end{equation*}

Since for $n>\EnOne$ the ratio $U/n$ is bounded by a constant, and even
monotone decreases to zero, as $n$ growth to infinity, the expressions
underlined by a brace and $(\NomiN{1}/n^2)/(\DeNomiN{1}/n)$ do not exceed a
constant for all $n>\EnOne$, as $U$ growth to infinity. The proof is completed
by summation, as it was done in Section~\ref{wertgjkty}:
$\IntegraL{1}{1}\sim\ln(U)U^{-1}$, $\IntegraL{1}{2}\sim\ln(U)U^{-1}$, and
$\IntegraL{1}{3}\sim\ln(U)U^{-1}$, as $U\to\infty$.

%%%%%%%%%%%%%%%%%%%%%%%%%%%%%%%%%%%%%%%%%%%%%%%%%%%%%%%%%%%%%%%%%%%%%%%%%%%%%%%%%%%%%%%%
\subsection{The asymptotic behavior of $\Integral{2}$}\label{qwerjkgulgh}
%%%%%%%%%%%%%%%%%%%%%%%%%%%%%%%%%%%%%%%%%%%%%%%%%%%%%%%%%%%%%%%%%%%%%%%%%%%%%%%%%%%%%%%%

Let us verify that for $R=\mcA\tfrac{U}{\sqrt{n}}+\mcB\sqrt{n}$,
$M=\mcC\tfrac{U}{\sqrt{n}}$,
$L=\tfrac{\mcC^2}{\mcA}\tfrac{U}{\sqrt{n}}-\mcB\sqrt{n}$ with
$\mcA,\mcB,\mcC>0$
\begin{multline*}
\Integral{2}=U\sum_{n>\frac{(\E{T})^2}{B_1}\,U^2}
n^{-3/2}\bigg(\int_{0}^{\infty}\big(2MR+y\big)^{-1}
\big(1+\sqrt{y}\big)^{-3}\,dy
\\[-8pt]
+\int_{0}^{-\frac{L}{2M}}\big(2MR-y\big)^{-1}
\big(1+\sqrt{y}\big)^{-3}\,dy\bigg)=\underline{O}\bigg(\dfrac{\ln(U)}{U}\bigg),\
U\to\infty.
\end{multline*}

Put the above $R$, $M$, and $L$ in the integrals evaluated in
Lemma~\ref{wqerthrjnmr}. For the first of them, it is checked by direct
calculations that
\begin{equation}\label{saesrfgewge}
\begin{aligned}
&\frac{U}{n^{3/2}}\frac{K-3}{(1+K)^2}\,\Big|_{K=2MR} \eqOK U\,\frac{2\mcC
U(\mcA U+\mcB n)-3n}{\sqrt{n}(2\mcC U(\mcA U+\mcB n)+n)^2},
\\[4pt]
&\frac{U}{n^{3/2}}\frac{\pi\sqrt{K}(K-3)}{(1+K)^3}\,\bigg|_{K=2MR}\eqOK
U^{3/2}\,\frac{\pi\sqrt{2\mcC(\mcA U+\mcB n)}\big(2\mcC U(\mcA U+\mcB
n)-3n\big)}{(2\mcC U(\mcA U+\mcB n)+n)^3},
\\[4pt]
&\frac{U}{n^{3/2}}\frac{3K-1}{(1+K)^3}\ln(K)\,\bigg|_{K=2MR}\eqOK
U\,\frac{\sqrt{n}(6\mcC U(\mcA U+\mcB n)-n)}{\big(2\mcC U(\mcA U+\mcB
n)+n\big)^3}\ln\left(\frac{2\mcC U(\mcA U+\mcB n)}{n}\right).
\end{aligned}
\end{equation}
For the second of them, it is checked by direct calculations that
\begin{equation}\label{qwergehn}
\begin{aligned}
&\frac{U}{n^{3/2}}\frac{4\sqrt{P}+(3+K)P}{(K-1)^2(1
+\sqrt{P})^2}\,\bigg|_{\substack{P=-\frac{L}{2M}\\K=2MR}}
\\[2pt]
&\hskip 40pt \eqOK U\,\frac{(\mcA\mcB n-\mcC^2U)(2\mcC U(\mcA U+\mcB n)+3n)
+4n\sqrt{2\mcA\mcC U(\mcA\mcB n-\mcC^2U)}}{\sqrt{n}\big(\sqrt{2\mcA\mcC U}
+\sqrt{\mcA\mcB n-\mcC^2 U}\big)^2\big(2\mcC U(\mcA U+\mcB n)-n\big)^2},
\\[2pt]
&\frac{U}{n^{3/2}}\frac{\sqrt{K}(3+K)}{(K-1)^3}\ln\left(\frac{\sqrt{K}+\sqrt{P}}{\sqrt{K}
-\sqrt{P}}\right)\,\bigg|_{\substack{P=-\frac{L}{2M}\\K=2MR}}
\\[2pt]
&\hskip 40pt\eqOK U^{3/2}\frac{\sqrt{2\mcC}\sqrt{\mcB n+\mcA U}\big(2\mcC
U(\mcA U+\mcB n)+3n\big)} {(2\mcC U(\mcA U+\mcB
n)-n)^3}\ln\left(-\frac{\NomiN{2}}{\DeNomiN{2}}\right),
\\[2pt]
&\frac{U}{n^{3/2}}\frac{3K+1}{(K-1)^3}\ln\left(\frac{K-P}{K(1+\sqrt{P})^2}
\right)\,\bigg|_{\substack{P=-\frac{L}{2M}\\K=2MR}} \eqOK
U\,\frac{\sqrt{n}(n+6\mcC U(\mcB n+\mcA U))}{(2\mcC U(\mcA U+\mcB
n)-n)^3}\ln\left(\frac{\NomiN{3}}{\DeNomiN{3}}\right),
\end{aligned}
\end{equation}
where
\begin{equation*}
\begin{aligned}
\NomiN{2}&=(\mcA\mcB n-\mcC^2U)n)^{1/2}+2(\mcA\mcC^2U^2(\mcA U+\mcB n))^{1/2},
\\
\DeNomiN{2}&=((\mcA\mcB n-\mcC^2U)n)^{1/2}-2(\mcA\mcC^2U^2(\mcA U+\mcB
n))^{1/2}
\end{aligned}
\end{equation*}
and
\begin{equation*}
\begin{aligned}
\NomiN{3}&=(\mcC^2U-\mcA\mcB n)n+4\mcA\mcC^2 U^2(\mcA U+\mcB n),
\\
\DeNomiN{3}&=(\mcB n+\mcA U)\mcC U(2\sqrt{\mcA\mcC U}+(2(\mcA\mcB
n-\mcC^2U))^{1/2})^2.
\end{aligned}
\end{equation*}

Using the standard technique of investigating the asymptotic behavior of the
summands in
$\Integral{2}=\IntegraL{2}{1}+\IntegraL{2}{2}+\IntegraL{2}{3}+\IntegraL{2}{4}
+\IntegraL{2}{5}+\IntegraL{2}{6}$ described in Section \ref{qasrdtjkyl}, we
have first\footnote{Note that in sums with $n>\frac{(\E{T})^2}{B_1}\,U^2$ the
ratio $U^2/n$ tends to zero, as $U\to\infty$.} $\IntegraL{2}{1}\sim U^{-1}$, as
$U\to\infty$, since
\begin{equation*}
\IntegraL{2}{1}=U\hskip -6pt\sum_{n>\frac{(\E{T})^2}{B_1}\,U^2}\frac{2\mcC
U(\mcA U+\mcB n)-3n}{\sqrt{n}(2\mcC U(\mcA U+\mcB n)+n)^2} =U\hskip
-6pt\sum_{n>\frac{(\E{T})^2}{B_1}\,U^2}\frac{1}{n^{3/2}}\underbrace{\frac{2\mcC
U(\mcA\frac{U}{n}+\mcB)-3} {(2\mcC U(\mcA\frac{U}{n}+\mcB)+1)^2}}_{\sim 1/U},
\end{equation*}
and similarly
\begin{equation*}
\begin{aligned}
\IntegraL{2}{2}&=U^{3/2}\hskip
-6pt\sum_{n>\frac{(\E{T})^2}{B_1}\,U^2}\frac{\pi\sqrt{2\mcC(\mcA U+\mcB n)}
\big(2\mcC U(\mcA U+\mcB n)-3n\big)}{(2\mcC U(\mcA U+\mcB n)+n)^3}
%\\
%&\hskip 80pt=U^{3/2}\,\sum_{n>\frac{(\E{T})^2}{B_1}\,U^2}\frac{1}{n^{3/2}}
%\underbrace{\frac{\pi\sqrt{2\mcC(\mcA\frac{U}{n}+\mcB)} \big(2\mcC
%U(\mcA\frac{U}{n}+\mcB)-3\big)}{(2\mcC
%U(\mcA\frac{U}{n}+\mcB)+1)^3}}_{\sim1/U^2}
\sim U^{-3/2},
\\
\IntegraL{2}{3}&=U\hskip
-4pt\sum_{n>\frac{(\E{T})^2}{B_1}\,U^2}\frac{\sqrt{n}(6\mcC U(\mcA U+\mcB
n)-n)}{\big(2\mcC U(\mcA U+\mcB n)+n\big)^3} \ln\left(\frac{2\mcC U(\mcA U+\mcB
n)}{n}\right)\sim\ln(U)U^{-2},
\\
\IntegraL{2}{4}&=U\hskip -4pt\sum_{n>\frac{(\E{T})^2}{B_1}\,U^2}\frac{(\mcA\mcB
n-\mcC^2U)(2\mcC U(\mcA U+\mcB n)+3n) +4n\sqrt{2\mcA\mcC U(\mcA\mcB
n-\mcC^2U)}} {\sqrt{n}\big(\sqrt{2\mcA\mcC U}+\sqrt{\mcA\mcB n-\mcC^2
U}\big)^2\big(2\mcC U(\mcA U+\mcB n)-n\big)^2} \sim U^{-1},
\\
\IntegraL{2}{5}&=U^{3/2}\hskip -4pt\sum_{n>\frac{(\E{T})^2}{B_1}\,U^2}
\frac{\sqrt{2\mcC}\sqrt{\mcB n+\mcA U}\big(2\mcC U(\mcA U+\mcB
n)+3n\big)}{(2\mcC U(\mcA U+\mcB
n)-n)^3}\ln\left(-\frac{\NomiN{2}}{\DeNomiN{2}}\right) \sim U^{-3/2},
\\
\IntegraL{2}{6}&=U\hskip -6pt\sum_{n>\frac{(\E{T})^2}{B_1}\,U^2}\hskip
-6pt\frac{\sqrt{n}(n+6\mcC U(\mcB n+\mcA U))}{(2\mcC U(\mcA U+\mcB
n)-n)^3}\ln\left(\frac{\NomiN{3}}{\DeNomiN{3}}\right) \sim\ln(U)U^{-2},
\end{aligned}
\end{equation*}
as $U\to\infty$. The proof is complete.

%%%%%%%%%%%%%%%%%%%%%%%%%%%%%%%%%%%%%%%%%%%%%%%%%%%%%%%%%%%%%%%%%%%%%%%%%%%%%%%%%%%%%%%%
\subsection{The asymptotic behavior of $\Integral{3}$}\label{werrtyjrtjrtrjk}
%%%%%%%%%%%%%%%%%%%%%%%%%%%%%%%%%%%%%%%%%%%%%%%%%%%%%%%%%%%%%%%%%%%%%%%%%%%%%%%%%%%%%%%%

Let us verify that
\begin{equation*}
\Integral{3}=U\sum_{n=\EnOne}^{\infty}n^{-3/2}\int_{L}^{\infty}(y+R)^{-1}(y^2+M^2)^{-3/2}dy
=\underline{O}\bigg(\dfrac{\ln(U)}{U}\bigg),\ U\to\infty,
\end{equation*}
where
\begin{equation}\label{sdwrtfyju}
\begin{aligned}
R&=\bigg(\frac{\mcA+\mcC
\cK{}}{1+\cK{2}}\bigg)\frac{U}{\sqrt{n}}+\frac{\mcB}{1+\cK{2}}\sqrt{n},\quad
%\\[2pt]
M=\bigg(\frac{\mcC-\mcA \cK{}}{1+\cK{2}}\bigg)\frac{U}{\sqrt{n}}-\frac{\mcB
\cK{}}{1+\cK{2}}\sqrt{n},
\\[2pt]
L&=\bigg(\frac{c\,\E{T}\mcC^2}{\E{Y}\mcA} -\frac{\cK{}(\mcC-\mcA
\cK{})}{1+\cK{2}} \bigg)\frac{U}{\sqrt{n}}-\frac{\mcB}{1+\cK{2}}\sqrt{n},
\end{aligned}
\end{equation}
i.e., for $R=\frac{\tilde{R}+\cK{}\tilde{M}}{1+\cK{2}}$,
$M=\frac{\tilde{M}-\cK{}\tilde{R}}{1+\cK{2}}$,
$L=\tilde{L}-\cK{}\frac{\tilde{M}-\cK{}\tilde{R}}{1+\cK{2}}$ with $\tilde{R}$,
$\tilde{M}$, $\tilde{L}$ defined in \eqref{sdfghffghjfgj} with
$\mcA=\frac{\E{T}}{\sqrt{B_1}}\frac{\sqrt{\D{Y}}}{c\sqrt{\D{T}}}>0$,
$\mcB=\frac{\sqrt{B_1}}{\sqrt{\D{Y}\D{T}}}>0$,
$\mcC=\frac{\E{Y}}{c\sqrt{B_1}}>0$.

Put these $R$, $M$, and $L$ into the integral evaluated in
Lemma~\ref{rthfgnmhfgj}. It is checked by direct calculations that\footnote{It
is noteworthy that $\mcC U-\cK{}(\mcB n+\mcA U)=(\mcC-\cK{}\mcA)U-\cK{}\mcB
n>0$ since $\cK{}<0$ (see \eqref{wqqgbdf}).}:
\begin{equation*}
\begin{aligned}
&\frac{Un^{-3/2}R}{M^2(R^2+M^2)} \eqOK\frac{U(1+\cK{2})^2(\mcB
n+(\mcA+\cK{}\mcC)U)}{ (\mcC U-\cK{}(\mcB n+\mcA U))^2((\mcB n+\mcA
U)^2+\mcC^2U^2)},
\\[10pt]
&\frac{Un^{-3/2}(M^2+LR)} {M^2\sqrt{L^2+M^2}\,(R^2+M^2)}
\myEq\frac{(1+\cK{2})^2U}{(\mcC U-\cK{}(\mcB n+\mcA U))^2\big((\mcB n+\mcA
U)^2+\mcC^2U^2\big)}\frac{\NomiN{4}}{\DeNomiN{4}},
\\[10pt]
&\frac{Un^{-3/2}}{(R^2+M^2)^{3/2}}
\ln\left(\frac{M^2-LR+\sqrt{L^2+M^2}\sqrt{R^2+M^2}}{(L+R)\,\big(\sqrt{R^2+M^2}-R\,\big)}\right)
\myEq\frac{U{(1+\cK{2})}^{3/2}\ln\left({\NomiN{5}}/{\DeNomiN{5}}\right)}{((\mcB
n+\mcA U)^2+\mcC^2U^2)^{3/2}},
\end{aligned}
\end{equation*}
where
\begin{equation*}
\begin{aligned}
\NomiN{4}&=\mcA\mcC^2U^2-(\mcB n+\mcA U)(\mcA\mcB n-c\tfrac{\E{T}}{\E{Y}}\mcC^2
U) +\cK{}(-4\mcA\mcB n-3\mcA^2 U+c\tfrac{\E{T}}{\E{Y}}\mcC^2 U)\mcC U
\\
&+\cK{2}(\mcA\mcB^2n^2+3\mcA^2\mcB nU+\mcB c\tfrac{\E{T}}{\E{Y}}n\mcC^2 U+2\mcA^3U^2+
\mcA c\tfrac{\E{T}}{\E{Y}}\mcC^2U^2-\mcA\mcC^2U^2)
\\
&+\cK{3}(\mcA^2+c\tfrac{\E{T}}{\E{Y}}\mcC^2)\mcC U^2,
\\[6pt]
\DeNomiN{4}&=\big(\big(\mcA(\mcB n+\cK{}\mcC U)-\mcA^2\cK{2}U
-c\tfrac{\E{T}}{\E{Y}}(1+\cK{2})\mcC^2 U\big)^2+\mcA^2(\mcC U-\cK{}(\mcB n+\mcA
U))^2\big)^{1/2}
\end{aligned}
\end{equation*}
and
\begin{equation*}
\begin{aligned}
\NomiN{5}&=\mcB(c\tfrac{\E{T}}{\E{Y}}\mcC^2-\mcA^2)nU+\mcA\mcC^2(c\tfrac{\E{T}}{\E{Y}}-1)U^2
-\mcA\mcB^2n^2+\cK{}(\mcA^2+c\tfrac{\E{T}}{\E{Y}}\mcC^2)\mcC U^2
\\
&-(1+\cK{2})^{-1/2}\big((\mcB n+\mcA U)^2+\mcC^2U^2\big)^{1/2}\big((
\mcA(\mcB n+\cK{}\mcC U)-\mcA^2\cK{2}U
\\
&-c\tfrac{\E{T}}{\E{Y}}(1+\cK{2})\mcC^2U)^2+\mcA^2(\cK{}(\mcB n+\mcA U)-\mcC U)^2\big)^{1/2},
\\[8pt]
\DeNomiN{5}&=(\mcA^2+c\tfrac{\E{T}}{\E{Y}}\mcC^2)U\big(\mcB n+\mcA U
-(1+\cK{2})^{-1/2}((\mcB n+\mcA U)^2+\mcC^2U^2)^{1/2}\big)
\\
&+\cK{}\mcC(\mcA^2+c\tfrac{\E{T}}{\E{Y}}\mcC^2)U^2
-\cK{2}(1+\cK{2})^{-1/2}(\mcA^2+c\tfrac{\E{T}}{\E{Y}}\mcC^2)
\big((\mcB n+\mcA U)^2+\mcC^2U^2\big)^{1/2}U.
\end{aligned}
\end{equation*}

\medskip We have
$\Integral{3}=\IntegraL{3}{1}+\IntegraL{3}{2}+\IntegraL{3}{3}$, where\medskip
\begin{equation*}
\begin{aligned}
\IntegraL{3}{1}&=U\sum_{n>\EnOne}\frac{\mcB n+(\mcA+\cK{}\mcC)U}{(\mcC
U-\cK{}(\mcB n+\mcA U))^2((\mcB n+\mcA U)^2+\mcC^2U^2)} \sim U^{-1},
\\[10pt]
\IntegraL{3}{2}&=U\sum_{n>\EnOne}\frac{1}{(\mcC U-\cK{}(\mcB n+\mcA
U))^2\big((\mcB n+\mcA
U)^2+\mcC^2U^2\big)}\left(\frac{\NomiN{4}}{\DeNomiN{4}}\right) \sim U^{-1},
\end{aligned}
\end{equation*}
and, using \eqref{sdfgrweghsfbh},
\begin{equation*}
\begin{aligned}
\IntegraL{3}{1}&=U\sum_{n>\EnOne}\frac{1}{\big((\mcB n+\mcA
U)^2+\mcC^2U^2\big)^{3/2}}\ln\left(\frac{\NomiN{5}}{\DeNomiN{5}}\right)
\sim\ln(U)U^{-1}.
\end{aligned}
\end{equation*}
The proof is complete.

%%%%%%%%%%%%%%%%%%%%%%%%%%%%%%%%%%%%%%%%%%%%%%%%%%%%%%%%%%%%%%%%%%%%%%%%%%%%%%%%%%%%%%%%
\subsection{The asymptotic behavior of $\Integral{4}$ and $\Integral{6}$}\label{saerdfgsfh}
%%%%%%%%%%%%%%%%%%%%%%%%%%%%%%%%%%%%%%%%%%%%%%%%%%%%%%%%%%%%%%%%%%%%%%%%%%%%%%%%%%%%%%%%

Recall (see \eqref{qwrethrj}) that
$\frac{1}{\E{Y}-c\E{T}}=\frac{\mcC-\cK{}\mcA}{\cK{}\mcB}$. The difference
between $\Integral{4}$, $\Integral{6}$ and $\Integral{3}$ lies only in the
range of summation: for $\Integral{4}$ it is
$\EnOne<n<\frac{U}{\E{Y}-c\E{T}}-\frac{K}{\cK{}\mcB}U$, and for $\Integral{6}$
it is $n>\frac{U}{\E{Y}-c\E{T}}+\frac{K}{\cK{}\mcB}U$. The same way as in
analyzing $\Integral{3}$, for $R$, $M$, $L$ set in \eqref{sdwrtfyju}, we turn
to the integral evaluated in Lemma~\ref{rthfgnmhfgj}, and the rest of the proof
consist in examining the asymptotic behavior, as $U\to\infty$, of the sums
similar to those in Section~\ref{werrtyjrtjrtrjk}, e.g., of
\begin{equation*}
\begin{aligned}
&U\sum_{\EnOne<n<\frac{\mcC-\cK{}\mcA}{\cK{}\mcB}U-\frac{K}{\cK{}\mcB}U}\frac{\mcB
n+(\mcA+\cK{}\mcC)U}{(\mcC U-\cK{}(\mcB n+\mcA U))^2((\mcB n+\mcA
U)^2+\mcC^2U^2)}\sim U^{-1},
\\
&U\sum_{n>\frac{\mcC-\cK{}\mcA}{\cK{}\mcB}U+\frac{K}{\cK{}\mcB}U}\frac{\mcB
n+(\mcA+\cK{}\mcC)U}{(\mcC U-\cK{}(\mcB n+\mcA U))^2((\mcB n+\mcA
U)^2+\mcC^2U^2)}\sim U^{-1}.
\end{aligned}
\end{equation*}
Leaving this checking to the reader, we point the main difference: in this case
$\cK{}>0$ and $(\mcC U-\cK{}(\mcB n+\mcA U))^2=((\mcC-\cK{}\mcA)U-\cK{}\mcB
n)^2$ vanishes for $n=\frac{\mcC-\cK{}\mcA}{\cK{}\mcB}U$. But both cases, for
$\Integral{4}$, since
$\EnOne<n<\frac{\mcC-\cK{}\mcA}{\cK{}\mcB}U-\frac{K}{\cK{}\mcB}U$, and for
$\Integral{6}$, since
$n>\frac{\mcC-\cK{}\mcA}{\cK{}\mcB}U+\frac{K}{\cK{}\mcB}U$, we have
\begin{equation*}
(\mcC U-\cK{}(\mcB n+\mcA U))^2=((\mcC-\cK{}\mcA)U-\cK{}\mcB n)^2>K^2U^2.
\end{equation*}

%%%%%%%%%%%%%%%%%%%%%%%%%%%%%%%%%%%%%%%%%%%%%%%%%%%%%%%%%%%%%%%%%%%%%%%%%%%%%%%%%%%%%%%%
\subsection{The asymptotic behavior of $\Integral{5}$}\label{fghrfhdfd}
%%%%%%%%%%%%%%%%%%%%%%%%%%%%%%%%%%%%%%%%%%%%%%%%%%%%%%%%%%%%%%%%%%%%%%%%%%%%%%%%%%%%%%%%

The investigation of the asymptotic behavior of
\begin{multline*}
\Integral{5}=U\hskip
-2pt\sum_{\frac{\mcC-\cK{}\mcA}{\cK{}\mcB}U-\frac{K}{\cK{}\mcB}U<n
<\frac{\mcC-\cK{}\mcA}{\cK{}\mcB}U+\frac{K}{\cK{}\mcB}U}\hskip -4pt
n^{-3/2}\Bigg(\int_{0}^{\infty}\big(2MR+y\big)^{-1}
\big(1+\sqrt{y}\big)^{-3}\,dy
\\[-8pt]
+\int_{0}^{-\frac{L}{2M}}\big(2MR-y\big)^{-1}
\big(1+\sqrt{y}\big)^{-3}\,dy\Bigg),
\end{multline*}
where $M$, $R$, and $L$ are defined in \eqref{sdwrtfyju}, is quite analogous to
investigation of the asymptotic behavior of $\Integral{2}$. We leave it to the
reader.

%%%%%%%%%%%%%%%%%%%%%%%%%%%%%%%%%%%%%%%%%%%%%%%%%%%%%%%%%%%%%%%%%%%%%%%%%%%%%%%%%%%%%%%%
\subsection{The asymptotic behavior of $\Integrall{1}$}\label{sadghjmgh}
%%%%%%%%%%%%%%%%%%%%%%%%%%%%%%%%%%%%%%%%%%%%%%%%%%%%%%%%%%%%%%%%%%%%%%%%%%%%%%%%%%%%%%%%

We have to verify that
\begin{equation*}
\begin{aligned}
\Integrall{1}&=U\hskip -4pt\sum_{\EnOne<n<\frac{B_2}{B_1}U}
n^{-2}\int_{L}^{\infty}|y|\big(y+R\big)^{-1}\big(y^2+M^2\big)^{-3/2}\,dy
\\[-6pt]
&+U\hskip-4pt\sum_{\frac{B_2}{B_1}U<n<\frac{(\E{T})^2}{B_1}\,U^2}\hskip -2pt
n^{-2}\int_{L}^{\infty}|y|\big(y+R\big)^{-1} \big(y^2+M^2\big)^{-3/2}\,dy
%\\[-8pt]
=\underline{O}\bigg(\dfrac{\ln(U)}{U}\bigg),\ U\to\infty,
\end{aligned}
\end{equation*}
for\footnote{In the first sum (with $n<\frac{B_2}{B_1}U$) we have $L\geqslant
0$, while in the second sum (with $n>\frac{B_2}{B_1}U$) we have $L\leqslant
0$.} $R=\mcA\tfrac{U}{\sqrt{n}}+\mcB\sqrt{n}$, $M=\mcC\tfrac{U}{\sqrt{n}}$,
$L=\tfrac{\mcC^2}{\mcA}\tfrac{U}{\sqrt{n}}-\mcB\sqrt{n}$ with
$\mcA,\mcB,\mcC>0$.

Put the above $R$, $M$, and $L$ in the first integral (wherein $L\geqslant 0$)
evaluated in Lemma~\ref{wqetfyuk6}. It is checked by direct calculations that
\begin{equation*}
\begin{aligned}
&\frac{Un^{-2}}{R^2+M^2}\eqOK\frac{U}{n((\mcB n+\mcA U)^2+\mcC^2U^2)},
\\
&\frac{Un^{-2}(R-L)}{(R^2+M^2)\sqrt{M^2+L^2}}\eqOK\frac{U(2\mcA\mcB
n+(\mcA^2-\mcC^2)U)}{n\big((\mcB n+ \mcA U)^2+\mcC^2U^2\big)\sqrt{(\mcA\mcB
n-\mcC^2U)^2+\mcA^2\mcC^2U^2}},
\\
&\frac{Un^{-2}R}{(R^2+M^2)^{3/2}}
\ln\left(\frac{(R+L)(\sqrt{R^2+M^2}-R)}{M^2-RL+\sqrt{R^2+M^2}\sqrt{M^2+L^2}}\right)
\eqOK\frac{U(\mcB n+\mcA U)\ln(\NomiN{6}/\DeNomiN{6})}{n((\mcB n+\mcA
U)^2+\mcC^2U^2)^{3/2}},
\end{aligned}
\end{equation*}
where
\begin{equation*}
\begin{aligned}
\NomiN{6}&=(\mcA^2+\mcC^2)U(((\mcB n+\mcA U)^2+\mcC^2U^2)^{1/2}-(\mcB n+\mcA
U)),
\\
\DeNomiN{6}&=\mcB(\mcA\mcB n-\mcC^2U)n+\mcA^2\mcB nU+((\mcB n+\mcA
U)^2+\mcC^2U^2)^{1/2}((\mcA\mcB n-\mcC^2U)^2+\mcC^2\mcA^2U^2)^{1/2}.
\end{aligned}
\end{equation*}

Put the above $R$, $M$, and $L$ in the second integral (wherein $L\leqslant 0$)
evaluated in Lemma~\ref{wqetfyuk6}. It is checked by direct calculations that
\begin{equation*}
\begin{aligned}
&\frac{Un^{-2}(2R+M)}{M(R^2+M^2)}\eqOK\frac{2\mcB
n+(2\mcA+\mcC)U}{n\mcC(\mcB^2n^2+ 2\mcA\mcB nU+(\mcA^2+\mcC^2)U^2)},
\\
&\frac{Un^{-2}(R-L)}{(R^2+M^2)\sqrt{M^2+L^2}}\eqOK\frac{U(2\mcA\mcB
n+(\mcA^2-\mcC^2)U)}{n\big((\mcB n+ \mcA U)^2+\mcC^2U^2\big)\sqrt{(\mcA\mcB
n-\mcC^2U)^2+\mcA^2\mcC^2U^2}},
\\
&\frac{Un^{-2}R}{(R^2+M^2)^{3/2}}\ln\left(\frac{R^2(\sqrt{R^2+M^2}-R)(M^2-RL+\sqrt{R^2+M^2}\sqrt{M^2+L^2})}
{M^2(M+\sqrt{R^2+M^2})^2(R+L)}\right)
\\
&\hskip 40pt\eqOK\frac{U(\mcB n+\mcA U)\ln(-\NomiN{7}/\DeNomiN{7})}{n((\mcB
n+\mcA U)^2+\mcC^2U^2)^{3/2}},
\end{aligned}
\end{equation*}
where
\begin{equation*}
\begin{aligned}
\NomiN{7}&=(\mcB n+\mcA U)^2(\mcB n+\mcA U-((\mcB n+\mcA
U)^2+\mcC^2U^2)^{1/2})(\mcA\mcB^2n^2+\mcA^2\mcB nU-\mcB\mcC^2 nU
\\
&+((\mcB n+\mcA U)^2+\mcC^2U^2)^{1/2}((\mcA\mcB
n-\mcC^2U)^2+\mcA^2\mcC^2U^2)^{1/2}),
\\
\DeNomiN{7}&=\mcC^2(\mcA^2+\mcC^2)U^3(\mcC U+((\mcB n+\mcA
U)^2+\mcC^2U^2)^{1/2})^2.
\end{aligned}
\end{equation*}

We have
$\Integrall{1}=\IntegralL{1}{1}+\IntegralL{1}{2}+\IntegralL{1}{3}+\IntegralL{1}{4}+\IntegralL{1}{5}+\IntegralL{1}{6}$,
where
\begin{equation*}
\begin{aligned}
\IntegralL{1}{1}&=U\sum_{\EnOne<n<\frac{B_2}{B_1}U}\frac{1}{n((\mcB n+\mcA
U)^2+\mcC^2U^2)} \sim U^{-1},
\\
\IntegralL{1}{2}&=U\sum_{\EnOne<n<\frac{B_2}{B_1}U}\frac{2\mcA\mcB
n+(\mcA^2-\mcC^2)U}{n\big((\mcB n+ \mcA U)^2+\mcC^2U^2\big)\sqrt{(\mcA\mcB
n-\mcC^2U)^2+\mcA^2\mcC^2U^2}}\sim U^{-1},
\\
\IntegralL{1}{*3}&=U\sum_{\EnOne<n<\frac{B_2}{B_1}U}\frac{(\mcB n+\mcA
U)\ln(\NomiN{6}/\DeNomiN{6})}{n((\mcB n+\mcA U)^2+\mcC^2U^2)^{3/2}}\sim
\ln(U)U^{-1},
\end{aligned}
\end{equation*}
and
\begin{equation*}
\begin{aligned}
\\
\IntegralL{1}{4}&=\sum_{\frac{B_2}{B_1}U<n<\frac{(\E{T})^2}{B_1}\,U^2}\frac{2\mcB
n+(2\mcA+\mcC)U}{n\mcC(\mcB^2n^2+ 2\mcA\mcB nU+(\mcA^2+\mcC^2)U^2)}\sim U^{-1},
\\
\IntegralL{1}{5}&=U\sum_{\frac{B_2}{B_1}U<n<\frac{(\E{T})^2}{B_1}\,U^2}\frac{2\mcA\mcB
n+(\mcA^2-\mcC^2)U}{n\big((\mcB n+ \mcA U)^2+\mcC^2U^2\big)\sqrt{(\mcA\mcB
n-\mcC^2U)^2+\mcA^2\mcC^2U^2}}\sim U^{-1},
\\
\IntegralL{1}{6}&=U\sum_{\frac{B_2}{B_1}U<n<\frac{(\E{T})^2}{B_1}\,U^2}\frac{(\mcB
n+\mcA U)\ln(\NomiN{6}/\DeNomiN{6})}{n((\mcB n+\mcA U)^2+\mcC^2U^2)^{3/2}}\sim
\ln(U)U^{-1}.
\end{aligned}
\end{equation*}
The proof is complete.

%%%%%%%%%%%%%%%%%%%%%%%%%%%%%%%%%%%%%%%%%%%%%%%%%%%%%%%%%%%%%%%%%%%%%%%%%%%%%%%%%%%%%%%%
\subsection{The asymptotic behavior of $\Integrall{2}$}\label{wertkktukt}
%%%%%%%%%%%%%%%%%%%%%%%%%%%%%%%%%%%%%%%%%%%%%%%%%%%%%%%%%%%%%%%%%%%%%%%%%%%%%%%%%%%%%%%%

We have to verify that
\begin{equation*}
\begin{aligned}
\Integrall{2}&=U\hskip -8pt\sum_{n>\frac{(\E{T})^2}{B_1}\,U^2}\hskip -4pt
n^{-2}\int_{L}^{\infty}|y|\big(y+R\big)^{-1}\big(1+(2M|y|)^{1/2}\big)^{-3}\,dy
=\underline{O}\bigg(\dfrac{\ln(U)}{U}\bigg),\ U\to\infty,
\end{aligned}
\end{equation*}
for\footnote{In the sum with $n>\frac{(\E{T})^2}{B_1}\,U^2$ we have $L<0$.}
$R=\mcA\tfrac{U}{\sqrt{n}}+\mcB\sqrt{n}$, $M=\mcC\tfrac{U}{\sqrt{n}}$,
$L=\tfrac{\mcC^2}{\mcA}\tfrac{U}{\sqrt{n}}-\mcB\sqrt{n}$ with
$\mcA,\mcB,\mcC>0$.

By making the change of variables, rewrite the integral in $\Integrall{2}$ as
\begin{equation*}
\frac{1}{2M}\int_{\frac{L}{2M}}^{\infty}|y|\big(y+2MR\big)^{-1}\big(1+\sqrt{|y|}\big)^{-3}\,dy
\end{equation*}
and recall that it is evaluated in explicit form in Lemma~\ref{eruyjtky}. Put
the above $R$, $M$, and $L$ into thus modified $\Integrall{2}$. It is checked
by direct calculations that
\begin{equation*}
\begin{aligned}
&\frac{U}{2Mn^2}\frac{5K+1}{(1+K)^2}\,\Big|_{K=2MR}\eqOK\frac{n+10\mcB\mcC nU+
10\mcA\mcC U^2}{2\sqrt{n}\mcC(n+2\mcB\mcC nU+2\mcA\mcC U^2)^2},
\\[6pt]
&\frac{U}{2Mn^2}\frac{\pi
K^{3/2}(K-3)}{(1+K)^3}\,\bigg|_{K=2MR}\eqOK\pi\sqrt{2}\frac{(\mcC U(\mcB n+\mcA
U))^{3/2}(2\mcA\mcC U^2+n(2\mcB\mcC U-3))}{n\mcC(n+2\mcB\mcC nU+2\mcA\mcC
U^2)^3},
\\[6pt]
&\frac{U}{2Mn^2}\frac{5K-1}{(K-1)^2}\,\bigg|_{K=2MR}\eqOK\frac{10\mcA\mcC U^2+
n(10\mcB\mcC U-1)}{2\sqrt{n}\mcC(2\mcA\mcC U^2+n(2\mcB\mcC U-1))^2},
\end{aligned}
\end{equation*}
\begin{equation*}
\begin{aligned}
\frac{U}{2Mn^2}&\frac{2\sqrt{P}(1-3K)-(5K-1)}{(K-1)^2(1+\sqrt{P})^2}
\,\bigg|_{\substack{P=-\frac{L}{2M}\\K=2MR}}
\\
&\hskip 30pt\eqOK-\frac{\sqrt{2}{\sqrt{\mcA U}}{\sqrt{\mcA\mcB n-\mcC^2 U}}
(6\mcC U(\mcA U+\mcB n)-n)}{\sqrt{n\mcC}(\sqrt{2\mcA\mcC U}+\sqrt{\mcA\mcB
n-\mcC^2 U})^2(2\mcC U(\mcA U+\mcB n)-n)^2}
\\
&\hskip 30pt+\frac{\mcA U(n-10\mcC U(\mcB n+\mcA U))}{\sqrt{n}(\sqrt{2\mcA\mcC
U}+\sqrt{\mcA\mcB n-\mcC^2U})^2(2\mcC U(\mcA U+\mcB n)-n)^2}
\end{aligned}
\end{equation*}
and (cf. \eqref{saesrfgewge} and \eqref{qwergehn})
\begin{equation*}
\begin{aligned}
&\frac{U}{2Mn^2}\frac{K(3K-1)}{(1+K)^3}\ln(K)\,\bigg|_{K=2MR}
\\
&\hskip 40pt\eqOK U\frac{\mcB n+\mcA U}{\sqrt{n}}\frac{(6\mcC U(\mcA U+\mcB
n)-n)}{\big(2\mcC U(\mcA U+\mcB n)+n\big)^3}\ln\left(\frac{2\mcC U(\mcA U+\mcB
n)}{n}\right),
\\[6pt]
&\frac{U}{2Mn^2}\frac{K^{3/2}(3+K)}{(K-1)^3}
\ln\left(\frac{\sqrt{P}+\sqrt{K}}{\sqrt{K}-\sqrt{P}}\right)\,\bigg|_{\substack{P=-\frac{L}{2M}\\K=2MR}}
\\
&\hskip 40pt\eqOK U^{3/2}\frac{\mcB+\mcA U}{n}\frac{\sqrt{2\mcC}\sqrt{\mcB
n+\mcA U}\big(2\mcC U(\mcA U+\mcB n)+3n\big)} {(2\mcC U(\mcA U+\mcB
n)-n)^3}\ln\left(-\frac{\NomiN{2}}{\DeNomiN{2}}\right),
\\[6pt]
&\frac{U}{2Mn^2}\frac{K(3K+1)}{(K-1)^3}\ln\left(\frac{K-P}
{K(1+\sqrt{P})^2}\right)\,\bigg|_{\substack{P=-\frac{L}{2M}\\K=2MR}}
\\
&\hskip 40pt\eqOK U\frac{\mcB+\mcA U}{\sqrt{n}}\frac{(n+6\mcC U(\mcB n+\mcA
U))}{(2\mcC U(\mcA U+\mcB
n)-n)^3}\ln\left(\frac{\NomiN{3}}{\DeNomiN{3}}\right).
\end{aligned}
\end{equation*}

We have $\Integrall{2}=\IntegralL{2}{1}+\IntegralL{2}{2}+\IntegralL{2}{3}
+\IntegralL{2}{4}+\IntegralL{2}{5}+\IntegralL{2}{6}+\IntegralL{2}{7}$, where,
e.g.,
\begin{equation*}
\begin{aligned}
\IntegralL{2}{1}&=\sum_{n>\frac{(\E{T})^2}{B_1}\,U^2}\frac{n+10\mcB\mcC nU+
10\mcA\mcC U^2}{2\sqrt{n}\mcC(n+2\mcB\mcC nU+2\mcA\mcC U^2)^2}\sim U^{-1},
\\[0pt]
\IntegralL{2}{2}&=\pi\sqrt{2}\sum_{n>\frac{(\E{T})^2}{B_1}\,U^2}\frac{(\mcC
U(\mcB n+\mcA U))^{3/2}(2\mcA\mcC U^2+n(2\mcB\mcC U-3))}{n\mcC(n+2\mcB\mcC
nU+2\mcA\mcC U^2)^3}\sim U^{-1},
\end{aligned}
\end{equation*}
and so on, so that $\Integrall{2}\sim \ln(U)U^{-1}$, $U\to\infty$, as required.
The proof is complete.

%%%%%%%%%%%%%%%%%%%%%%%%%%%%%%%%%%%%%%%%%%%%%%%%%%%%%%%%%%%%%%%%%%%%%%%%%%%%%%%%%%%%%%%%
\subsection{Asymptotics for integrals of rational and exponential functions}\label{qweryherh}
%%%%%%%%%%%%%%%%%%%%%%%%%%%%%%%%%%%%%%%%%%%%%%%%%%%%%%%%%%%%%%%%%%%%%%%%%%%%%%%%%%%%%%%%

The following integrals of rational and exponential functions can be evaluated
in explicit form in terms of $\Gamma(0,-\tfrac12
x^2)=\int_{-x^2/2}^{\infty}t^{-1}e^{-t}dt$ and
$\text{Erfi}(x)=\frac{2}{\sqrt{\pi}}e^{x^2}D(x)$, where
$D(x)=e^{-x^2}\int_{0}^{x^2}e^{t^2}dt$ is Dawson function. Recall that for
$x>0$
\begin{equation}\label{wsdfghfn}
D(x)=\frac{1}{2x}+\frac{1}{4x^3}+\frac{3}{8x^5}+\dots,\quad x\to+\infty,
\end{equation}
and
\begin{equation}\label{wqthdjnr}
\Gamma(0,-\tfrac12 x^2)=(-\tfrac12 x^2)^{-1}e^{\tfrac12
x^2}\Big[1+\frac{2}{x^2}+\frac{8}{x^4}+\dots\Big],\quad x\to+\infty,
\end{equation}
so that $\tfrac12 e^{-\frac12 x^2}\Gamma(0,-\frac12 x^2)\sim-x^{-2}$, as
$x\to+\infty$.

\begin{lemma}\label{rtyi6k6k}
For $R>0$, we have
\begin{equation*}
\int_{0}^{\infty}(y+R)^{-1}\exp\big\{-\tfrac{1}{2}\,y^2\big\}dy=\underline{O}(R^{-1}),\quad
R\to\infty.
\end{equation*}
\end{lemma}

\begin{proof}
Bearing in mind \eqref{wsdfghfn} and \eqref{wqthdjnr}, for $R\to\infty$ we
have\footnote{Note that $\text{Erfi}(R/\sqrt{2})=\frac{2}{\sqrt{\pi}}e^{\frac12 R^2}D(R/\sqrt{2})$.}
\begin{multline*}
\int_{0}^{\infty}(y+R)^{-1}\exp\big\{-\tfrac{1}{2}\,y^2\big\}dy=\tfrac{1}{2}
e^{-\frac12 R^2}\big[\pi\text{Erfi}\big(\tfrac{1}{\sqrt{2}}R\big)+\,\Gamma(0,-\tfrac12
R^2)-\ln(R)\big]
\\
=\sqrt{\pi}D\big(\tfrac{1}{\sqrt{2}}R\big)+\tfrac{1}{2}e^{-\frac12
R^2}\big[\Gamma(0,-\tfrac12 R^2)-\ln(R)\big]\sim\tfrac{\sqrt{\pi}}{\sqrt{2}}R^{-1},
\end{multline*}
which completes the proof.
\end{proof}

\begin{lemma}\label{6uiytkty}
For $R>0$, we have
\begin{equation*}
\int_{0}^{\infty}y\,(y+R)^{-1}\exp\big\{-\tfrac{1}{2}\,y^2\big\}dy=\underline{O}(R^{-1}),\quad
R\to\infty.
\end{equation*}
\end{lemma}
\begin{proof}
Bearing in mind \eqref{wsdfghfn} and \eqref{wqthdjnr}, for $R\to\infty$ we have
\begin{multline*}
\int_{0}^{\infty}y\,(y+R)^{-1}\exp\big\{-\tfrac{1}{2}\,y^2\big\}dy=\tfrac{1}{2}e^{-\frac12
R^2}\big[e^{\frac12 R^2}\sqrt{2\pi}
-\pi R\,\text{Erfi}\big(\tfrac{1}{\sqrt{2}}R\big) -R\,\Gamma(0,-\tfrac12R^2)
\\
+2R\ln(R)\,\big]=\tfrac{\sqrt{\pi}}{\sqrt{2}}\big[1-\sqrt{2}R\,D\big(\tfrac{1}{\sqrt{2}}R\big)\,\big]
-\tfrac{R}{2}e^{-\frac12R^2}\big[\Gamma(0,-\tfrac12R^2)-2\ln(R)\,\big]\sim
R^{-1},
\end{multline*}
which completes the proof.
\end{proof}

Both Lemmas~\ref{rtyi6k6k} and \ref{6uiytkty} are proved by means of very
precise calculations which yield exactly the main terms of approximations,
rather than investigate the asymptotic behavior. Investigation of mere the
asymptotic behavior can be done much simpler (see, e.g., \citeNP{[De Bruijn
1958]}). Indeed, the function $\exp\{-\tfrac12 y^2\}$ is concentrated in a
narrow region around the origin. For the remaining factor in the integrand,
note that in this region $0<K_1 R^{-1}\leqslant\big(y+R\big)^{-1}\leqslant K_2
R^{-1}$. Routine estimation completes the proof. Using these considerations, it
is easy to check the following lemma.

\begin{lemma}\label{wertejht}
For $L+R>0$, we have
\begin{equation*}
\begin{gathered}
\int_{L}^{\infty}(y+R)^{-1}\exp\big\{-\tfrac{1}{2}\,y^2\big\}dy=\underline{O}(R^{-1}),\quad
R\to\infty,
\\
\int_{L}^{\infty}|y|(y+R)^{-1}\exp\big\{-\tfrac{1}{2}\,y^2\big\}dy=\underline{O}(R^{-1}),\quad
R\to\infty.
\end{gathered}
\end{equation*}
\end{lemma}

We merely note that the integrand in Lemma~\ref{wertejht} is positive and has
no singularities within the range of integration $[L,\infty)$. Indeed, the
point of singularity $y=-R$ of $(y+R)^{-1}$ lies outside the range of
integration since $-R<L$.

%%%%%%%%%%%%%%%%%%%%%%%%%%%%%%%%%%%%%%%%%%%%%%%%%%%%%%%%%%%%%%%%%%%%%%%%%%%%%%%%%%%%%%%%
\subsection{The asymptotic behavior of $\Jntegral{1}$}\label{werfgwegfwe}
%%%%%%%%%%%%%%%%%%%%%%%%%%%%%%%%%%%%%%%%%%%%%%%%%%%%%%%%%%%%%%%%%%%%%%%%%%%%%%%%%%%%%%%%

Let us verify that
\begin{equation*}
\Jntegral{1}=U\hskip
-6pt\sum_{\EnOne<n<\frac{(\E{T})^2}{B_1}U^2}n^{-3/2}\exp\big\{-\tfrac12
M^2\big\}\int_{L}^{\infty}\big(y+R\big)^{-1}\exp\big\{-\tfrac12\,y^2\big\}\,dy=\underline{O}(U^{-1}),\
U\to\infty,
\end{equation*}
for $R=\mcA\tfrac{U}{\sqrt{n}}+\mcB\sqrt{n}$, $M=\mcC\tfrac{U}{\sqrt{n}}$,
$L=\tfrac{\mcC^2}{\mcA}\tfrac{U}{\sqrt{n}}-\mcB\sqrt{n}$ with
$\mcA,\mcB,\mcC>0$. We use Lemma~\ref{wertejht}, note that
\begin{equation*}
\frac{Un^{-3/2}}{R}\exp\big\{-\tfrac12 M^2\big\}=\frac{U}{n\big(\mcA U+\mcB
n\big)} \exp\Big\{-\frac{C^2U^2}{2n}\Big\}
\end{equation*}
and that
\begin{multline*}
U\sum_{\EnOne<n<\frac{(\E{T})^2}{B_1}U^2}\frac{1}{n\big(\mcA U+\mcB n\big)}
\exp\Big\{-\frac{C^2U^2}{2n}\Big\}
\\[-8pt]
=\frac{1}{U}\sum_{\EnOne<n<\frac{(\E{T})^2}{B_1}U^2}\frac{\frac{U^2}{n^2}}{\mcA
\frac{U}{n}+\mcB}
\exp\Big\{-\frac{C^2}{2}\frac{U^2}{n}\Big\}=\underline{O}(U^{-1}),
\end{multline*}
as $U\to\infty$, as required.

%%%%%%%%%%%%%%%%%%%%%%%%%%%%%%%%%%%%%%%%%%%%%%%%%%%%%%%%%%%%%%%%%%%%%%%%%%%%%%%%%%%%%%%%
\subsection{The asymptotic behavior of $\Jntegral{2}$}\label{ertherhrhr}
%%%%%%%%%%%%%%%%%%%%%%%%%%%%%%%%%%%%%%%%%%%%%%%%%%%%%%%%%%%%%%%%%%%%%%%%%%%%%%%%%%%%%%%%

We have to check that
\begin{equation*}
\Jntegral{2}=U\sum_{n>\frac{(\E{T})^2}{B_1}U^2}n^{-3/2}
\int_{L}^{\infty}\big(y+R\big)^{-1}\exp\big\{-\tfrac12\,y^2\big\}\,dy=\underline{O}(U^{-1}),\
U\to\infty,
\end{equation*}
for $R=\mcA\tfrac{U}{\sqrt{n}}+\mcB\sqrt{n}$ with $\mcA,\mcB>0$. The same way
as in Section \ref{werfgwegfwe}, we have
\begin{multline*}
U\sum_{n>\frac{(\E{T})^2}{B_1}U^2}n^{-3/2}\big(\mcA\tfrac{U}{\sqrt{n}}+\mcB\sqrt{n}\big)^{-1}
\exp\left\{-\mcC^2\,\tfrac{U^2}{n}\right\}
\\[-4pt]
\leqslant
U\sum_{n>\frac{(\E{T})^2}{B_1}U^2}\frac{1}{n^2}\frac{\exp\{-\mcC^2\,\tfrac{U^2}{n}\}}{1+\frac{U}{n}}
\leqslant
KU\sum_{n>\frac{(\E{T})^2}{B_1}U^2}\frac{1}{n^2}=\underline{O}(U^{-1}),
\end{multline*}
as $U\to\infty$, as required.


\begin{thebibliography}{99}
\baselineskip=12pt

{\normalsize

\bibitem[Bhattacharya and Ranga Rao (1976)]{[Bhattacharya Rao 1976]}
\hskip -20pt\textsc{Bhattacharya,\,R.N., and Ranga Rao,\,R.} (1976)
\emph{Normal Approximation and Asymptotic Expansions.} Wiley \& Sons, New York,
etc.

\bibitem[Borovkov and Dickson (2008)]{[Borovkov Dickson 2008]}
\hskip -20pt\textsc{Borovkov,\,K., and Dickson,\,D.C.M.} (2008) On the ruin
time distribution for a Sparre Andersen process with exponential claim sizes,
{\IME}, Vol.~42, 1104--1108.

\bibitem[De Bruijn (1958)]{[De Bruijn 1958]}
\hskip -20pt\textsc{De Bruijn,\,N.G.} (1958) \emph{Asymptotic methods in
analysis.} North-Holland.

\bibitem[Dubinskaite (1982)]{[Dubinskaite 1982]}
\hskip -20pt\textsc{Dubinskaite,\,J.} (1982) Limit theorems in $\R^k$.~I,
{\LMJ}, Vol.~22, No.~2, 129--140, doi:10.1007/BF00969611.

\bibitem[Gradshtein and Ryzhik (1980)]{[Gradshtein 1980]}
\hskip -20pt\textsc{Gradshtein,\,I.S., and Ryzhik,\,I.M.} (1980) \emph{Table of
Integrals, Series, and Products.} Academic Press, New York.

\bibitem[Kendall (1957)]{[Kendall 1957]}
\hskip -20pt\textsc{Kendall,\,D.G.} (1957) Some problems in the theory of dams,
{\JRSS}, Vol.~19, 207--212.

\bibitem[Malinovskii (1993)]{[Malinovskii 1993]}
\hskip -20pt\textsc{Malinovskii,\,V.K.} (1993) Limit theorems for stopped
random sequences. I: rates of convergence and asymptotic expansions, {\TPA},
Vol.~38, 673--693.

\bibitem[Nagaev (1965)]{[Nagaev 1965]}
\hskip -20pt\textsc{Nagaev,\,S.V.} (1965) Some limit theorems for large
deviations, {\TPA}, Vol.~10, 214--235.

\bibitem[Nagaev (1979)]{[Nagaev 1979]}
\hskip -20pt\textsc{Nagaev,\,S.V.} (1979) Large deviations of sums of
independent random variables, {\ANP}, Vol.~7, 754--789.

\bibitem[Petrov (1975)]{[Petrov 1975]}
\hskip -20pt\textsc{Petrov,\,V.V.} (1975) \emph{Sums of Independent Random
Variables}. Springer, Berlin, etc.

\bibitem[Petrov (1995)]{[Petrov 1995]}
\hskip -20pt\textsc{Petrov,\,V.V.} (1995) \emph{Limit Theorems of Probability
Theory. Sequences of Independent Random Variables}. Clarendon Press, Oxford
Studies in Probability.

\bibitem[Karatsuba and Voronin (1992)]{[Karatsuba Voronin 1992]}
\hskip -20pt\textsc{Karatsuba,\,A.A., and Voronin,\,S.M.} (1992) \emph{The
Riemann Zeta-Function}. Walter de Gruyter. 

}\end{thebibliography}
\end{document}